\documentclass[11pt]{amsart}
 \pdfoutput=1
\usepackage{tikz-cd}
\usepackage{mathtools}
\usepackage{enumerate}
\usepackage{color}
\usepackage{verbatim}
\usepackage[abs]{overpic}
\usepackage{hyperref}
\usepackage{amsthm,amsfonts,amsmath,amscd,amssymb}
\usepackage{xcolor}
\usepackage{graphics,graphicx}
\usepackage{caption}
\usepackage{subcaption}
\let\oldmarginpar\marginpar
\renewcommand\marginpar[1]{\-\oldmarginpar[\raggedleft\footnotesize #1]%
{\raggedright\footnotesize #1}}
\theoremstyle{plain}
\newtheorem{thm}{Theorem}[section]
\newtheorem{conj}[thm]{Conjecture}
\newtheorem{lemma}[thm]{Lemma}
\newtheorem{prop}[thm]{Proposition}
\newtheorem{cor}[thm]{Corollary}

\theoremstyle{definition}

\newtheorem{definition}[thm]{Definition}
\newtheorem{example}[thm]{Example}

\newtheorem{remark}[thm]{Remark}
\newtheorem{construction}[thm]{Construction}

\theoremstyle{remark}

\numberwithin{equation}{section}

\newcommand{\N}{\mathbb{N}}
\newcommand{\Z}{\mathbb{Z}}

\newcommand{\R}{\mathbb{R}}
\newcommand{\C}{\mathbb{C}}

\renewcommand{\a}{\alpha}
\newcommand{\La}{\Lambda}
\newcommand{\la}{\lambda}
\newcommand{\p}{\varphi}
\newcommand{\e}{\varepsilon}
\newcommand{\dd}{\partial}

\newcommand{\sse}{\subseteq}
\newcommand{\lr}{\longrightarrow}

\newcommand{\im}{\operatorname{im}}
\newcommand{\id}{\operatorname{id}}
\newcommand{\ob}{\operatorname{ob}}

\newcommand{\Diff}{\operatorname{Diff}}
\newcommand{\Symp}{\operatorname{Symp}}

\newcommand{\Sk}{\text{Sk}}

\newcommand{\st}{\text{st}}
\newcommand{\std}{\text{st}}

\def\Op{{\mathcal O}{\it p}\,}
\begin{document}
\begin{abstract}
In this article we show that in any dimension there exist infinitely many pairs of formally contact isotopic isocontact embeddings into the standard contact sphere which are not contact isotopic. This is the first example of rigidity for contact submanifolds in higher dimensions. The contact embeddings are constructed via contact push-offs of higher-dimensional Legendrian submanifolds, a construction that generalizes the union of the positive and negative transverse push-offs of a Legendrian knot to higher dimensions. 
\end{abstract}

\title[Non-isotopic contact push-offs]{Non-simplicity of isocontact embeddings\\ in all higher dimensions}

\author{Roger Casals}
\address{University of California Davis, Dept. of Mathematics, Shields Avenue, Davis, CA 95616, USA}
\email{casals@math.ucdavis.edu}

\author{John B.~Etnyre}
\address{School of Mathematics, Georgia Institute
of Technology, 686 Cherry Street,  Atlanta, GA 30332-0160, USA}
\email{etnyre@math.gatech.edu}

\maketitle
\vspace{-0.6cm}
\section{Introduction}\label{sec:intro}

Transverse knots in contact 3--manifolds \cite{Etnyre05, Geiges08} have proven themselves instrumental to study the contact geometry of 3--manifolds. The first existence results for 3-dimensional contact structures \cite{Lutz71, Martinet71} crucially used transverse knots, and they were responsible for the birth of the tight and overtwisted dichotomy \cite{Bennequin83, Eliashberg89}. In fact, the collection of transverse representatives of fibered knots with pseudo-Anosov monodromy in a contact 3--manifold uniquely determines the contact structure up to contact isotopy \cite{EtnyreVanHorn-Morris11}.

The first instance of formally contact isotopic transverse knots, and thus smoothly isotopic knots, which were not contact isotopic, was discovered by J.S.~Birman and W.~Menasco \cite[Theorem 3]{BirmanMenasco06II}. This continued with \cite[Theorem 1.7]{EtnyreHonda05} where it was shown that the $(2,3)$-cable of the $(2,3)$-torus knot is not transversely simple. These results established the rigidity of transverse knots in 3-dimensional contact manifolds \cite{Eliashberg15}. Techniques and results on transverse knots have since developed \cite{EkholmEtnyreNgSullivan13a, Ng11, OzsvathSzaboThurston08, Plamenevskaya06} %\cite{EENS, Ng2, OST, Pl} 
yielding a wide range of applications and classifications results in low-dimensional contact geometry \cite{EtnyreLafountainTosun12, EtnyreNgVertesi13, Plamenevskaya06a}.

%Transverse knots in a contact 3--manifold are isocontact embeddings of the unique contact $(S^1,\xi_\st)$ into a contact 3--manifold.

In parallel, the theory of contact submanifolds in higher dimensions has a central role in the development of higher-dimensional contact and symplectic topology \cite{BormanEliashbergMurphy,CasalsMurphyPresas,Giroux02}. In the symplectic context, the classification of symplectic surfaces in $(\mathbb{C}\mathbb{P}^2,\omega_{\st})$ is an outstanding open problem, known as the symplectic isotopy problem \cite{AurouxSmith,Gromov85}. In addition, the first constructions of contact and symplectic submanifolds employed approximately holomorphic techniques \cite{DonaldsonAHT,IMP}. Remarkably, the study of isocontact embeddings in higher dimensions has seen recent developments, including \cite{EtnyreFurukawa17, EtnyreLekili18} for new results on isocontact embeddings into 5-dimensional contact manifolds, \cite{Lazarev,PancholiPandit} for existence results on codimension-2 contact embeddings, and \cite{BormanEliashbergMurphy15, CasalsMurphyPresas} for the study of overtwisted isocontact embeddings. Note that the Isotopy Extension Theorem \cite[Theorem 2.6.12]{Geiges08} holds for contact submanifolds in all dimensions, and it is thus equivalent to consider isotopies of isocontact embeddings and ambient contact isotopies.

The central question that has remained open, in dimension 5 and above, is the existence of (formally contact isotopic) isocontact embeddings which are not contact isotopic. In particular, there was no known higher-dimensional counterpart to the well-known rigidity for transverse knots \cite{Bennequin83,BirmanMenasco06II,Eliashberg15} in contact 3-manifolds. The aim of this article is to positively answer this question in all higher dimensions, thus establishing the rigidity of contact submanifolds in all dimensions. The first result we present is:

\begin{thm}\label{thm:main}
Let $(S^{2n+1},\xi_\std)$ be the standard contact sphere, and $(Y,\xi)$ the contact boundary of $(T^*S^n,\la_\st)$. Then there are two codimension-2 isocontact embeddings
$$i_0,i_1:(Y,\xi)\lr(S^{2n+1},\xi_\std),$$
such that $i_0$ and $i_1$ are formally contact isotopic and $i_0$ and $i_1$ are not contact isotopic.
\end{thm}

%For instance, in the 5-dimensional case this yields two non-isotopic embeddings of $(\R\P^3,\xi_\std)$ into $(S^5,\xi_std)$, where this contact structure $\xi_\std$ on $\R\P^3$ is the unique tight contact structure on $\R\P^3$.

Let us refer to an isotopy class of smooth embeddings as {\em simple} if any two isocontact embeddings in this smooth isotopy class, which are formally contact isotopic, are isotopic through isocontact embeddings. In this language, Theorem \ref{thm:main} is the first example of contact non-simplicity in higher dimensions. The three-dimensional case was first resolved by J.S. Birman and W. Menasco in \cite[Theorem 3]{BirmanMenasco06II} and \cite[Theorem 1.7]{EtnyreHonda05}.

\begin{remark}
The h-principle for isocontact embeddings \cite[Theorem 12.3.1]{EliashbergMishachev02} shows that formally isotopic contact embeddings in at least codimension-4 must be contact isotopic. In this sense, Theorem \ref{thm:main} is sharp and the codimension-2 condition cannot be removed.\hfill$\Box$
\end{remark}

The isocontact embeddings in Theorem~\ref{thm:main} are obtained through a construction which we refer to as the contact push-off of a Legendrian submanifold, defined in Section~\ref{sec:pushoff}. This technique yields a productive connection between the theory of higher-dimensional Legendrian submanifolds \cite{CasalsMurphy, EkholmEtnyreSullivan05b, EkholmEtnyreSullivan05a} and the study of isocontact embeddings.

%In short, given a Legendrian $\La\sse(Y,\xi)$ in a contact manifold $(Y,\xi)$ there is a canonical embedding of a neighborhood of the zero section of $T^*\La$ into $(Y,\xi)$ such that a contact form for $\xi$ restricts to the canonical 1-form on $T^*\La$. The boundary of this neighborhood is a codimension-2 contact submanifold of $(Y,\xi)$, which we refer to as the contact push-off of $\La$. In dimension 3, the contact push-off recovers the union of the positive and negative transverse push-offs \cite{Etnyre05}.

%We can also construct contact push-offs of singular Legendrian submanifolds. In Section \ref{sec:covers} we develop results in higher-dimensional contact surgery theory which also apply to such singular Legendrian submanifolds and we prove the following stronger result.

Theorem~\ref{thm:main} is proven in three steps. First, we construct the isocontact embeddings $i_0$ and $i_1$ using the contact push-off of Legendrian submanifolds. Second, we develop higher-dimensional contact surgery diagrams for the 2-fold contact branched covers of $(S^{2n+1},\xi_\std)$ branched along the images of these isocontact embeddings. Finally, the theory of pseudo-holomorphic curves, in the context of positive $S^1$-equivariant symplectic homology \cite[Section 2]{BourgeoisOancea} and \cite{KwonVanKoert16,Zhou}, allows us to distinguish the higher-dimensional contact structures given by these contact surgery diagrams.

\begin{remark}
 Theorem~\ref{thm:contactsurgery} proven in the course of the proof of Theorem \ref{thm:main} and Section~\ref{sec:covers} are of independent interest for the study of higher-dimensional Legendrian handlebody decompositions and contact surgeries, since they provide an explicit computational mechanism to obtain contact surgery diagrams for cyclic contact branch covers along contact push-offs in all higher-dimensions.\hfill$\Box$
 \end{remark}

In addition to Theorem~\ref{thm:main}, we use the contact push-offs of non-isotopic higher-dimensional singular Legendrian submanifolds in order to prove the following:

\begin{thm}\label{cor:infinitelymany}
Let $(S^{2n+1},\xi_\std)$ be the standard contact sphere. There exist infinitely many non-simple classes of isocontact embeddings in $(S^{2n+1},\xi_\std)$.
\end{thm}

In addition, should we assume the Legendrian surgery formula \cite[Theorem 5.4]{BourgeoisEkholmEliashberg12} and the well-definedness of the terms therein, the proof of Theorem \ref{thm:main} should yield infinitely many formally isotopic contact embeddings of a given contact manifold into $(S^{2n+1},\xi_\std)$ which are not contact isotopic. Thus, we state it as the following:

\begin{conj}\label{conj:infinite}
Let $(S^{2n+1},\xi_\std)$ be the standard contact sphere. There exist infinitely many formally contact isotopic isocontact embeddings of $(\dd T^*S^n,\xi_\std)$ into $(S^{2n+1},\xi_\std)$ which are not contact isotopic.\hfill$\Box$
\end{conj}

% ------------------------------------------------------------------------------------------------------
{\bf Organization.} The article is organized as follows. We begin with some background material in Section~\ref{sec:background}. Section~\ref{sec:pushoff} gives a method to construct codimension-2 contact embeddings into a given contact manifold via higher-dimensional Legendrian submanifolds. Section \ref{sec:covers} studies contact surgery diagrams for the contact cyclic branched covers along these contact submanifolds. Then Section~\ref{sec:mainproof} proves Theorem~\ref{thm:main}, and Section~\ref{sec:infinte} proves Theorem~\ref{cor:infinitelymany}.\\

% ------------------------------------------------------------------------------------------------------
{\bf Acknowledgements.} We are thankful to the referee for their detailed report. We are also grateful to them for suggesting the current argument for Lemma \ref{lem:formaldetermined}. We are grateful to Jo Nelson and Jeremy Van Horn Morris for useful discussions. R.~Casals is supported by the NSF grant DMS-1841913 and a BBVA Research Fellowship. J.~Etnyre is partially supported by the NSF grant DMS-1608684.\hfill$\Box$
% ------------------------------------------------------------------------------------------------------
% ------------------------------------------------------------------------------------------------------
\section{Preliminaries}\label{sec:background}

In this section we introduce the preliminaries required for this article. The reader is referred to \cite{ArnoldGivental01, Geiges08} for a good introduction to the basics of contact topology, and \cite{Etnyre05} for the theory of 3-dimensional transverse knots.

% ------------------------------------------------------------------------------------------------------
\subsection{Formal isotopies}\label{ssec:formal}

The condition for a smooth submanifold to be a Legendrian or a contact submanifold of a given contact manifold $(Y,\xi)$ can be sifted into two pieces, algebraic topology and differential geometry. This is the tenet of the h-principle \cite{EliashbergMishachev02, Gromov86} and allows us to focus on the genuinely geometric, instead of homotopical, features of contact topology.

For instance, the fact that two smoothly isotopic Legendrian knots in $(S^3,\xi_\st)$ with different Thurston-Bennequin invariants \cite[Section 2.6.1]{Etnyre05} are not Legendrian isotopic can be detected with classical homotopy theory, and we thus consider it an algebraic topological matter, rather than differential geometric.

The notion of a formal Legendrian submanifold is formalized as follows. 

\begin{definition}[\cite{EliashbergMishachev02,Murphy??}]\label{def:formalLeg}
Let $(Y,\xi)$ be a $(2n+1)$-dimensional contact manifold and $\La$ an $n$-dimensional smooth manifold. A formal Legendrian embedding of $\La$ into $(Y,\xi)$ is the data of the following commutative diagram
\[ \begin{tikzcd}
T\La \arrow{r}{F_s} \arrow[swap]{d}{\pi} & TY \arrow{d}{\pi} \\%
\La \arrow{r}{f}& Y
\end{tikzcd}
\]
where $s\in[0,1]$, the map $\pi$ is the canonical projection of the smooth tangent bundle, and the pair $(f,F_s)$ satisfies the following two properties:
\begin{itemize}
 \item[1.] $f$ is a smooth embedding and $F_0=df$,  
\item[2.] $F_s$ is a fiberwise injective bundle map covering $f$ for all $s\in[0,1]$, and
 \item[3.] $F_1(T_p\La)$ is a Lagrangian subspace of $\xi_{f(p)}\sse TY$, for all $p\in\La$. \hfill$\Box$
\end{itemize}
\end{definition}

For instance, in the case of $(Y,\xi)=(\R^3,\xi_\std)$, the set of connected components in the space of formal Legendrian knots is indexed by the smooth topological type of the knot, the rotation number, and the Thurston-Bennequin invariant \cite{Etnyre05}. This result and the higher-dimensional analogues for stably parallelizable Legendrians in $(Y,\xi)=(\R^{2n+1},\xi_\std)$ are discussed in \cite[Proposition~A.2]{Murphy??}.

The formal avatar of a codimension-2 contact submanifold is described as follows.

\begin{definition}[\cite{EliashbergMishachev02}]\label{def:formalcont}
Let $M$ be a $(2n-1)$-dimensional smooth manifold, a pair $(\eta,\omega)$ is said to be a formal (or almost) contact structure on $M$ if $\eta\sse TM$ is a codimension-1 sub-bundle and $\omega\in\Omega^2(M)$ is a $2$-form such that the conformal class of its restriction $\omega|_\eta$ to the sub-bundle $\eta\sse TM$ induces a conformal symplectic structure $(\eta,\omega)$ on the sub-bundle $\eta$.

Let $(Y,\xi,\tau)$ be a $(2n+1)$-dimensional formal contact manifold, $(M,\eta,\omega)$ a $(2n-1)$-dimensional formal contact manifold. A formal (iso)contact embedding of $(M,\eta,\omega)$ into $(Y,\xi)$ is the data of the following commutative diagram
\[ \begin{tikzcd}
TM \arrow{r}{F_s} \arrow[swap]{d}{\pi} & TY \arrow{d}{\pi} \\%
M \arrow{r}{f}& Y
\end{tikzcd}
\]
where $s\in[0,1]$ and the pair $(f,F_s)$ satisfies the following two properties:
\begin{itemize}
 \item[1.] $f$ is a smooth embedding and $F_0=df$,  
\item[2.] $F_s$ is a fiberwise injective bundle map covering $f$ for all $s\in[0,1]$, and
 \item[3.] $F_1$ induces a conformally symplectic map from $(\eta,\omega)$ into $(\xi,\tau)$.\hfill$\Box$
\end{itemize}
\end{definition}

In the case that $(Y,\xi)$ is a contact manifold, we shall always consider $\tau=d\alpha$, where the $1$-form $\alpha$ is a contact form $\xi=\ker(\alpha)$. Equipped with Definition \ref{def:formalcont}, two isocontact embeddings $i_0,i_1:(M,\xi)\lr(Y,\xi)$ are said to be formally contact isotopic if there exists a family
$$i_t:(M,\xi)\lr(Y,\xi),\quad t\in[0,1],$$
of formal contact embeddings. The central theme of this article is to show that there exist isocontact embeddings $i_0$ and $i_1$, and families of formal contact embeddings $\{i_s\}_{s\in[0,1]}$ which cannot be deformed to a family of isocontact embeddings for all $s\in[0,1]$ relative to the two given endpoints $s\in\{0,1\}=\dd[0,1]$.

% ------------------------------------------------------------------------------------------------------
\subsection{Legendrian submanifolds}\label{ssec:leg}

In this article, we will consider Legendrian submanifolds in a Darboux ball and describe them through their fronts \cite[Section 3.1]{Arnold90}. Consider the standard contact structure 
$$(\R^{2n+1},\xi_\st)=\left(\R^{2n+1},\ker\left\{dz-\sum_{i=1}^n y_i\, dx_i\right\}\right),$$
where $(x_1,y_1,\ldots,x_n,y_n,z)\in\R^{2n+1}$ are Cartesian coordinates. The projection map
$$\pi:\R^{2n+1}\lr \R^{n+1},\quad (x_1,y_1,\ldots,x_n,y_n,z)\longmapsto(x_1,\ldots,x_n,z)$$
is a Legendrian fibration and thus it can be used as a front projection \cite[Chapter 5]{ArnoldGivental01}. In particular, any Legendrian submanifold $\La\sse(\R^{2n+1},\xi_\std)$ is determined by its front projection $\pi(\La)\sse\R^{n+1}$. In this article any depiction of a Legendrian submanifold is a depiction of its front. Fronts can also be used in more general contact manifolds, see \cite[Section 2]{CasalsMurphy} for a detailed discussion and proofs of their validity.

\subsubsection{Legendrian stabilization} The operation of Legendrian stabilization has been detailed in \cite[Section~4.3]{EkholmEtnyreSullivan05a}, this operation is a crucial ingredient in our construction of formally contact isotopic non-isotopic contact embeddings. See also \cite{CieliebakEliashberg12,Murphy??}.

In order to describe this operation, consider the Legendrian $\La_{0}\sse\R^{2n+1}$ whose front $\pi_0$ is given by the disjoint union of the graphs of the two constant functions $0$ and $1$. 
\[
0,1:\R^n(x_1,\ldots,x_n)\lr\R(z).
\] 
This front is depicted in the left of Figure \ref{stabilizefig}. In addition to the Legendrian $\La_0$, we consider a second Legendrian whose front is described as follows.

Let $M\sse \R^n(x_1,\ldots,x_n)$ be a bounded smooth submanifold and consider a smooth function
$$f:\R^n(x_1,\ldots,x_n)\lr\R(z),$$
such that $1$ is a regular value, $f\cong 1.5$ in a neighborhood $\Op(M)$, and $f\cong 0$ away from $\Op(M)$. The Legendrian $\La_M\sse(\R^{2n+1},\xi_\st)$ is defined by having the Legendrian front $\pi_M$ given by the graph of the function $f$ and the function constant equal to $1$. Figure \ref{stabilizefig} on the right depicted this front $\pi_M$ in the case $n=2$ and $M=S^1$. The operation of Legendrian stabilization is the content of the following definition. 

\begin{definition}\label{def:legstab}
Let $\La\sse(Y,\xi)$ be a Legendrian submanifold and let $(B,\xi_\st)\sse(Y,\xi)$ be a Darboux ball such that the $(\La\cap B,B)$ is contactomorphic to $(\La_0,(\R^{2n+1},\xi_\st))$. The $M$-stabilization of the Legendrian $\La$ is the Legendrian submanifold obtained by replacing the intersection $(\La\cap B,B)$ with $(\La_M,\R^{2n+1})$.\hfill$\Box$

\begin{figure}[ht]{\tiny
\begin{overpic}%[grid,tics=10] 
{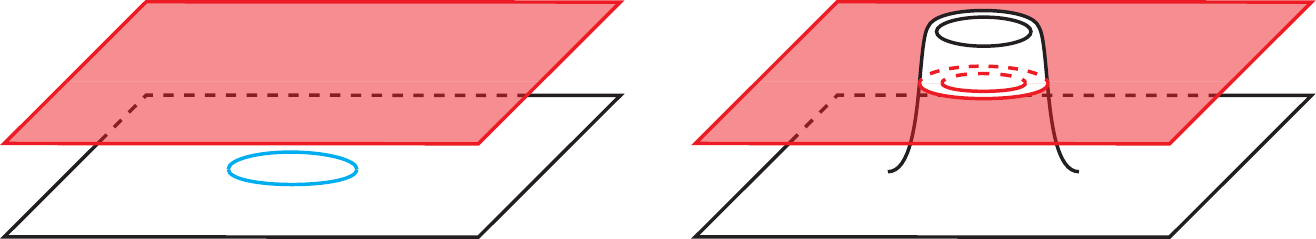}
%\put(120, 120){$+$}
\end{overpic}}
\caption{On the left is the front projection of the Legendrian submanifold $\La_0$ of $\R^{2n+1}$ with an embedded circle $M$ on the lower sheet. On the right is the resulting front $\La_M$ after an $M$-stabilization along that circle.}\label{stabilizefig}
\end{figure}
\end{definition}

Since $(S^{2n+1},\xi_\std)$ minus a point is contactomorphic to the standard contact structure $(\R^{2+1},\xi_\std)$, we can similarly define stabilizations of Legendrian submanifolds in $(S^{2n+1},\xi_\std)$.

In the present article, we will use the following form of stabilization. Given $\La\sse(S^{2n+1},\xi_\std)$, a Darboux ball $(B,\xi_\st)$ and a front for $\La\cap B\sse (B,\xi_\st)$, let $V\sse\R^{n+1}$ be a neighborhood of a point on a cusp edge in this front projection. Then the stabilization of the Legendrian $\La$ along an $S^1$ localized at $V$ will be denoted by $s(\La)$.

The results from \cite[Proposition 4.5]{EkholmEtnyreSullivan05a} and \cite[Appendix A]{Murphy??} combine to establish the following
\begin{lemma}[\cite{EkholmEtnyreSullivan05a}]\label{lem:formallegisotopy}
Let $\La\sse(S^{2n+1},\xi_\std)$ be a Legendrian submanifold, then the stabilized Legendrian $s(\La)\sse(S^{2n+1},\xi_\std)$ is formally Legendrian isotopic to $\La$. 
\end{lemma}
In contrast to Lemma \ref{lem:formallegisotopy}, pseudo-holomorphic invariants from symplectic field theory \cite{EkholmEtnyreSullivan05a,EliashbergGiventalHofer00} prove that there exist Legendrians $\La$ such that $\La$ is not Legendrian isotopic to its Legendrian stabilization $s(\La)$.

Finally, we shall use the following (non-standard) definition of a {\it loose} Legendrian, which will feature in the discussion of flexible Weinstein manifolds.

\begin{definition}[\cite{CieliebakEliashberg14,Murphy??}]\label{def:loose}
A Legendrian submanifold $\La$ of $(Y,\xi)$ is said to be loose if there exists a Legendrian $\La'$ such that $\La=s(\La')$.\hfill$\Box$
\end{definition}

\subsubsection{Legendrian connected sum}\label{sssec:legsum} The statement of Theorem \ref{thm:contactsurgery} describes the contact surgery diagram of a contact branch cover in terms of the higher-dimensional Legendrian connected sum of two Legendrian submanifolds. Let us include here the definition of the Legendrian connected sum \cite{Arnold90,BourgeoisSabloffTraynor15}.

Let $(Y,\xi)$ be a contact manifold and $\La_1,\La_2\sse(Y,\xi)$ two Legendrian submanifolds. For the purposes of the present work, let us assume that there exist two disjoint Darboux balls $(B_1,\xi_\st)$ and $(B_2,\xi_\st)$ respectively containing $\La_1$ and $\La_2$. Let $(B,\xi_\st)$ be the Darboux ball obtained by taking the standard neighborhood of $(B_1,\xi_\st)$ and $(B_2,\xi_\st)$ connected with any isotropic arc. The latter Darboux ball $(B,\xi_\st)$ allows us to consider $\La_1$ and $\La_2$ as Legendrians in $(\R^{2n+1},\xi_\st)$ such that their fronts in $\R^{n+1}(x_1,\ldots,x_n,z)$ are separated by the hyperplane $\{x_1=0\}$.\\

Consider an isotropic arc $\gamma\sse \R^{2n+1}$ with endpoints in $\La_1$ and $\La_2$, and assume that its front projection admits a tubular neighborhood whose boundary is Legendrian isotopic to the front projection of the $(n-1)$-dimensional standard Legendrian unknot in $\R^{n}$ times a linear interval. This will be referred to as a standard tubular neighborhood, for its existence see \cite{Geiges08}. Note that isotropic arcs $\gamma\sse \R^{2n+1}$ are classified by the $h$-principle for isotropic embeddings \cite{EliashbergMishachev02}.

\begin{definition}\label{def:legsum}
	Let $\La_1,\La_2\sse(Y,\xi)$ be two Legendrians and $\gamma\sse Y$ an arc with endpoints in $\La_1$ and $\La_2$ such that:
	\begin{itemize}
		\item[1.] $\overset{\circ}{\gamma}\sse Y\setminus\La_1\sqcup\La_2$, i.e. $\gamma$ only intersects the Legendrians at the endpoints.
		\item[2.] The points of intersection $\gamma\cap\La_1$ and $\gamma\cap\La_2$ belong to a cusp edge of the fronts of $\La_1$ and $\La_2$.
	\end{itemize}

The Legendrian connected sum of $\La_1$ and $\La_2$ along $\gamma$ is the Legendrian $\La_1\#_\gamma\La_2\sse(Y,\xi)$ whose front is obtained by removing the front projection of the neighborhoods $\Op(\gamma\cap\La_1)\cap\La_1$ and $\Op(\gamma\cap\La_2)\cap\La_2$ and concatenating with the boundary of the standard tubular neighborhood of $\gamma$.\hfill$\Box$
\end{definition}

The hypothesis that the points of intersection $\gamma\cap\La_1$ and $\gamma\cap\La_2$ belong to a cusp edge \cite{Arnold90} implies that the complements of $\Op(\gamma\cap\La_1)\cap\La_1$ and $\Op(\gamma\cap\La_2)\cap\La_2$ in the respective fronts have boundaries equal to the front of the  $(n-1)$-dimensional standard Legendrian unknot, thus making the concatenation with the standard tubular neighborhood well-defined. In using Definition \ref{def:legsum}, we will specify the precise isotropic arc $\gamma$ along which we perform the Legendrian connected sum.

\begin{remark}
 In the articles \cite[Section 4]{BourgeoisSabloffTraynor15} and  \cite{EkholmEtnyreSullivan05b} the Legendrian connected sum is implicitly constructed in the context of index 1 exact Lagrangian cobordisms, and thus Legendrian $0$-surgeries. In the classical theory of Legendrian singularities, this is the index one $A_2$-perestroika \cite[Figure 48]{Arnold90}.\hfill$\Box$
\end{remark}

Let $\La_0,\La_1\sse (Y,\xi)$ be two unlinked Legendrians, with $\La_0\sse(Y,\xi)$ loose. It follows from Definition \ref{def:legsum} and the $h$-principle for loose Legendrians \cite[Theorem 1.2]{Murphy??} that the Legendrian connected sum $\La_0\#\La_1\sse(Y,\xi)$ is a loose Legendrian. This indeed uses that $\La_0,\La_1$ exist within two disjoint Darboux balls, as in Definition \ref{def:legsum}. This will be used in Section \ref{sec:mainproof}.

 % ------------------------------------------------------------------------------------------------------
 \subsection{Weinstein Hypersurfaces}\label{ssec:WeinsteinHyper}

Let $(W,\la)$ be a Liouville manifold, i.e.\ an exact symplectic manifold $(W,\omega)$ with a choice of primitive $\la$, such that $d\la=\omega$. The vector field $X_\la$, $\omega$-dual to $\la$ is called the Liouville vector field, and in this article we assume that it is a complete vector field. A Weinstein manifold is a Liouville manifold together with a Lyapunov function for $X_\la$. The central feature of Weinstein structures is that the Lyapunov function gives a Legendrian handlebody decomposition of $W$ \cite{Eliashberg90a,Weinstein91}, capturing the symplectic topology of $(W,\la)$. In this section we review the basic ingredients in the study of Weinstein hypersurfaces. The reader is referred to \cite{CasalsMurphy,CieliebakEliashberg12} and \cite{Eliashberg18} for accounts on the general theory of Weinstein manifolds.

The notion of a Weinstein hypersurface of a contact manifold is introduced in \cite[Definition 1.3]{Avdek}, in the context of Liouville hypersurfaces, and further discussed by Y.~Eliashberg as part of a Weinstein pair in \cite[Section 2.1]{Eliashberg18}.

\begin{definition}\label{def:WeinsteinHyp}
Let $(Y,\ker{\a})$ be a contact manifold, a Weinstein hypersurface $\Sigma\sse Y$ is a codimension-1 submanifold such that $(\Sigma,\a|_\Sigma)$ is compatible with a Weinstein structure of $\Sigma$.\hfill$\Box$
\end{definition}

Let $(\Sigma,\a|_\Sigma)\sse(Y,\ker{\a})$ be a Weinstein hypersurface, then $(\dd\Sigma,\a|_{\dd\Sigma})\sse(Y,\xi)$ is a contact submanifold. The image of the isocontact embeddings in Theorem \ref{thm:main} are codimension-2 contact submanifolds bounded by a Weinstein hypersurface. There are many examples of contact submanifolds that do not bound Weinstein hypersurfaces \cite{Fran2,NiVan,Fran1}.

Given two Weinstein hypersurfaces $\Sigma_0,\Sigma_1\sse(Y,\xi)$, the Weinstein hypersurface sum $$(Y,\xi){}_{\Sigma_0}{\#}{_{\Sigma_1}}(Y,\xi)$$
of $(Y,\xi)$ along $\Sigma_0$ and $\Sigma_1$ is constructed in \cite[Section 3]{Avdek}. This is the contact manifold obtained by gluing the two contact complements
$$(Y\setminus\Op(\Sigma_0),\xi),\quad (Y\setminus\Op(\Sigma_1),\xi)$$
along their boundaries $\dd\Op(\Sigma_0)\sqcup\dd\Op(\Sigma_1)$. In order to perform a contact gluing along a hypersurface in a contact manifold, we use the technology of {\it convex hypersurfaces}. In short, Moser's method \cite[Section 4]{Moser1965} shows that a Weinstein hypersurface has a contact neighborhood of the form

$$(\Sigma\times[-\epsilon, \epsilon], \ker\{\a-dt\}),$$

where $t\in[-\epsilon, \epsilon]$ is a coordinate, see \cite[Lemma~3.1]{Avdek}. This standard neighborhood can be used to show that there exists a contact neighborhood $\Op(\Sigma)$ with a canonical convex hypersurface boundary $\dd\Op(\Sigma)$. Then the standard contact gluing along convex hypersurfaces is used to build $(Y,\xi){}_{\Sigma_0}{\#}{_{\Sigma_1}}(Y,\xi)$. We refer the reader to the work of E.~Giroux \cite[Section I]{Giroux91} for convex gluing in the three-dimensional case, and \cite[Section~3.3]{Avdek} and \cite[Section 2]{HondaHuang} for the higher-dimensional theory.

\begin{remark}
The Weinstein pages of adapted open books for contact manifolds $(Y_1,\xi_1)=\ob(W_1,\p_1)$ and $(Y_2,\xi_2)=\ob(W_2,\p_2)$ are convex hypersurfaces. If the Weinstein pages are Weinstein isomorphic, then it is possible to glue these adapted open books \cite{Avdek} along their Weinstein pages. Adapted open books will be further discussed in Section \ref{ssec:ob}.\hfill$\Box$
\end{remark}

\subsubsection{Cobordism description of Weinstein hypersurface sums}\label{sssec:sumcob} In this article, we shall use the following result, stated in \cite[Theorem 1.9]{Avdek}. We include our proof of the result since we will use the details of our proof below. If $(W,\la)$ is a Weinstein manifold, then we denote by $(W^{(1)},\la^{(1)})$ the Weinstein cobordism obtained by taking the product of $(W,\la)$ with a 2-dimensional Weinstein 1-handle $(D^1\times D^1,\la_\st)$.
%Weinstein cobordism $(W^{(1)},\la^{(1)})$ in the statement of Proposition \ref{prop:weinsteinsumcob} is the Weinstein product of $(W,\la)$ and a 2-dimensional Weinstein 1-handle $(D^1\times D^1,\la_\st)$. 
We will momentarily provide more details for $(W^{(1)},\la^{(1)})$ and prove Proposition~\ref{prop:weinsteinsumcob}.
%where the Weinstein cobordism $(W^{(1)},\la^{(1)})$ is defined right after the statement.

\begin{prop}\label{prop:weinsteinsumcob}
	Let $(Y_0,\xi_0)$ and $(Y_1,\xi_1)$ be two contact manifolds and
	$$i_0:(W,\la)\lr(Y_0,\xi_0),\quad i_1:(W,\la)\lr(Y_1,\xi_1)$$
	be two Weinstein hypersurface embeddings. Then the Weinstein sum
	$$(Y_0,\xi){}_{im(i_0)}{\#}{_{im(i_1)}}(Y_1,\xi)$$
	is the convex boundary of the Weinstein cobordism obtained by attaching the Weinstein cobordism $(W^{(1)},\la^{(1)})$ to the convex part of the symplectization of $(Y_0,\xi_0)\sqcup(Y_1,\xi_1)$ along the two embeddings of the Weinstein hypersurface $(W,\la)$.
	
	Furthermore, given a Weinstein handle decomposition of $(W,\la)$ we can explicitly describe the Legendrian attaching spheres for a Weinstein cobordism from $(Y_0,\xi_0)\sqcup(Y_1,\xi_1)$  to their Weinstein sum $(Y_0,\xi){}_{im(i_0)}{\#}{_{im(i_1)}}(Y_1,\xi)$.
\end{prop}

\begin{remark} Technically, \cite[Theorem 1.9]{Avdek} contains the first result in this direction, showing that the Weinstein connected sum yields a Weinstein cobordism. However, \cite{Avdek} does not provide an explicit description of the contact surgery diagram for the cobordism - though some of the provided examples allude to how this is done in dimension 3 - and we need the general improvement given by the proof of Proposition~\ref{prop:weinsteinsumcob} and Section~\ref{sec:covers} for our results.\hfill$\Box$
\end{remark}

Let $(W,\la,\p)$ be a Weinstein manifold and $H^{(1)}=(D^1\times D^1,\la_\st,\p_\st)$ a 2-dimensional Weinstein 1-handle \cite[Section 3]{Weinstein91}. Specifically, consider the $1$-form $\la_\st=2s\, dt + t\, ds$ where $(s,t)$ are coordinates on the product $D^1\times D^1$ and the Morse function $\p_\st=-s^2+t^2$.
Endow the product $W\times(D^1\times D^1)$ with its natural Weinstein structure \cite[Example 11.12.(3)]{CieliebakEliashberg12} given by $\la+\la_\st$ and $\p+\p_\st$, and denote the resulting Weinstein manifold by $(W^{(1)},\la^{(1)},\p^{(1)})$. Observe that if $v$ is the Liouville field associated to the Liouville 1-form $\la$ then $v+2s\frac{\partial}{\partial s} - t\frac{\partial}{\partial t}$ will be the Liouville field for the sum $\la+\la_\st$. The product structure induces the following decomposition of the smooth boundary of $(W^{(1)},\la^{(1)},\p^{(1)})$:
\[
\dd W^{(1)}= \left[\dd W\times (D^1\times D^1)\right] \sqcup \left[W\times ((D^1\times S^0)\sqcup(S^0\times D^1))\right].
\]
Let the piece $(S^0\times D^1)\sse H^{(1)}$ be the attaching region of the Weinstein 1-handle $H^{(1)}$, it is endowed with the contact structure $(J^1S^0,\xi_\st)$ where $S^0$ is the isotropic attaching sphere. The Weinstein manifold $(W^{(1)},\la^{(1)},\p^{(1)})$ can be understood as a Weinstein cobordism with concave contact boundary $(W\times (S^0\times D^1),\ker(\la+\la_\st))$, convex contact boundary
$$((\partial W\times (D^1\times D^1)) \cup (W\times (D^1\times S^0)),\ker(\la+\la_\st))$$
and with corners along $\dd W\times \dd (D^1\times D^1)$.

\begin{proof}[Proof of Proposition \ref{prop:weinsteinsumcob}] The contact structure induced on each of the two connected components of $(W\times (D^1\times S^0),\xi_{\st})$ by the Liouville form $\la^{(1)}$ is contactomorphic to the contactization of the Weinstein structure $(W,\la)$. Thus, in the hypothesis of Proposition \ref{prop:weinsteinsumcob} we can identify the concave end of $(W^{(1)},\la^{(1)},\p^{(1)})$ with contact neighborhoods of the two Weinstein hypersurfaces $im(i_0)\sqcup im(i_1)\sse (Y_0,\xi_0)\sqcup(Y_1,\xi_1)$.
	
	In addition, given the Weinstein structure of the Weinstein handle model \cite{Weinstein91}, we can identify the remaining $D^1$-direction in the Weinstein 1-handle $H^{(1)}$ with the symplectization Liouville direction of the symplectization of the contact manifold $(Y_0,\xi_0)\sqcup(Y_1,\xi_1)$. In this manner we attach the Weinstein cobordism $(W^{(1)},\la^{(1)})$ to the symplectization of $(Y_0,\xi_0)\sqcup(Y_1,\xi_1)$ along $(W,\la)$. Since the local model of a Weinstein sum is obtained by identifying the contact neighborhoods along the contactization direction we conclude that the Weinstein sum along $(W,\la)$ is equivalently obtained by exchanging the concave contact boundary $W\times ((S^0\times D^1)$ of the Weinstein cobordism $(W^{(1)},\la^{(1)},\p^{(1)})$ by its convex contact boundary $(\partial W\times (D^1\times D^1)) \cup (W\times (D^1\times S^0))$. This proves the first part of the statement of Proposition \ref{prop:weinsteinsumcob}.
	
Suppose now that $(W,\la)$ is presented as a Weinstein handlebody, we shall now show how to build the symplectic cobordism $(\overline{W},\overline{\la})$ above from a piece of the symplectization of  $(Y_0,\xi_0)\sqcup(Y_1,\xi_1)$ by attaching a Weinstein $(k+1)$-handle for each $k$-handle in the decomposition of $(W,\la)$. The construction is described as follows.

\begin{construction}\label{con}
Consider a Weinstein handle decomposition of $(W,\la)$, where the Weinstein handles are attached in order of increasing index, and let $h^k_1,\ldots, h_{i_k}^k$ be the $k$-handles in the give Weinstein handlebody decomposition. Let us denote by $W_k$ the union of all Weinstein handles of index less than or equal to $k$, and we emphasize that $h^{k}$ is attached to $\partial W_{k-1}$ along an isotropic $(k-1)$-sphere in $\partial W_{k-1}$. It should also be noted \cite{Weinstein91} that the 1-form $\lambda$ vanishes on the core disk $D^k_i$ of $h^k_i$ and thus the image of $D_i^k$ in $W_{k-1}$ will be an isotropic submanifold of $(Y_0,\xi_0)\sqcup(Y_1,\xi_1)$.

For each $0$-handle $h^0_i$, $1\leq i\leq i_0$, the image of the core in the hypersurfaces $\Sigma_0$ and $\Sigma_1$ gives rise to two points in $(Y_0,\xi_0)\sqcup(Y_1,\xi_1)$, one point belonging to each component of this disjoint union decomposition. This will be the attaching neighborhood of a Weinstein 1-handle for $W$.  Inductively, let us assume that we have attached all the $k$-handles of $\overline{W}$ resulting in $\overline{W}_k=([0,1]\times M) \cup (D^1 \times D^1 \times W_{k-1})$, where the components of $S^0 \times D^1 \times W_{k-1}$ have been glued to their images in $Y_0$ and $Y_1$ and the upper boundary of the cobordism is $(Y_0,\xi_0)\sqcup(Y_1,\xi_1)$ minus two copies of the standard neighborhoods $N_{W_{k-1}}$ of the $W_{k-1}$, union $D^1\times (\partial (D^1 \times W_{k-1}))$. Notice that $\partial (D^1\times W_{k-1})$ is a convex surface and the contact structure on $D^1\times (\partial (D^1\times W_{k-1}))$ is simply an $D^1$--invariant neighborhood of $\partial (D^1 \times W_{k-1})$.

Let us then explain how to attach the Weinstein $(k+1)$-handles. For each $k$-handle $h^{k}_i$ of $W$, $1\leq i\leq i_k$, the image of its core $C^{k}_i$ in $Y_0$ and $Y_1$ gives rise to two isotropic disks in $(Y_0,\xi_0)\sqcup(Y_1,\xi_1)$, and the attaching sphere $\partial C_i^k$ gives an isotropic annulus $D^1\times \partial C_i^k$ in  $D^1 \times (\partial W_{k-1})$. The union of the disks and annulus will give an isotropic $S^k$ in the upper boundary of $\overline{W}_k$, and the given trivializations of the conformal normal bundle to the $C^k_i$  extend over $D^1 \times \partial C_i^k$. We attach a Weinstein $(k+1)$-handle to $W_k$ along this sphere. \hfill$\Box$
\end{construction}

Finally, the above construction does recover $(\overline{W},\overline{\la})$ since the model for a Weinstein $k$-handle times $H^{(1)}$ with the vector field and Liouville form from above is the model of a Weinstein $(k+1)$-handle, and the attaching isotropic sphere is precisely as indicated. This concludes the proof of Proposition \ref{prop:weinsteinsumcob}.
\end{proof}

\subsubsection{Legendrian lifts} Let $L\sse(\Sigma,\a)$ be an exact Lagrangian in a Weinstein hypersurface $(\Sigma,\a)\sse(Y,\xi)$, with primitive $f:L\lr[-\epsilon, \epsilon]$ such that $\a|_{L}=df$. Consider the standard neighborhood $(\Sigma\times[-\epsilon, \epsilon], \ker\{\a-dt\})$, then
\[
i:L\lr\Sigma\times[-\epsilon, \epsilon]: p\longmapsto (p,f(p))
\] 
is a Legendrian embedding. By definition, any Legendrian representative $\La\sse(Y,\xi)$ of the isotopy type of the image $\im(L)$ is said to be a {\em Legendrian lift} of $L\sse(\Sigma,\a)$ into $(Y,\xi)$.

The case where the primitive $f\equiv 0$ is identically zero is also of interest, as for the zero section of a cotangent bundle, in which case the Lagrangian submanifold itself is a Legendrian submanifold, when regarded in $(Y,\xi)$. We refer to \cite[Section 2]{ArnoldGivental01,CasalsMurphy} for discussions and details on Legendrian lifts.

\subsubsection{Flexible Weinstein Manifolds} Flexible Weinstein fillings are introduced in \cite[Chapter 11.8]{CieliebakEliashberg12}. By definition \cite[Definition 11.29]{CieliebakEliashberg12}, a $2n$-dimensional Weinstein manifold is {\em flexible} if it admits a Weinstein handlebody decomposition \cite{CasalsMurphy} such that all the critical $n$--handles are attached along loose Legendrian spheres, as introduced in Definition~\ref{def:loose} above. The only currently available method to detect flexibility of a Weinstein manifold is to draw a Weinstein handlebody diagram \cite[Section 2]{CieliebakEliashberg12} exhibiting the loose charts for the attaching Legendrian submanifold. This is the method used in \cite[Section 4]{CasalsMurphy} to give many examples of affine manifolds whose underlying Weinstein structure is flexible, and also the method we will use in Section \ref{sec:mainproof} to exhibit flexible Weinstein fillings, which is crucial in our proof of Theorem \ref{thm:main}.
% ------------------------------------------------------------------------------------------------------
\subsection{Contact branched covers}\label{ssec:contactbranch}
The proof of Theorem \ref{thm:main} and Theorem~\ref{cor:infinitelymany} also rely on the understanding of higher-dimensional contact branched covers, which we review in this section. The operation of a contact branched cover in 3-dimensional contact topology was first introduced J.~Gonzalo in \cite{Gonzalo87}, and contact branched covers in the higher dimension were constructed by H.~Geiges \cite{Geiges97}, as follows.
 
Let $C\sse(Y,\ker\{\a\})$ be a codimension-2 contact submanifold and $p:Y(C)\lr Y$ a smooth branched cover with branched locus $C\sse Y$. The local model for the smooth branched covering along $C$ is given by

$$P:\R^n\times\C\lr\R^n\times\C: (x,z)\lr (x,z^n),$$

where $C$ is locally identified as $\R^n\times\{0\}\sse \R^n\times\C$. See \cite{EtnyreFurukawa17,Rolfsen76} for more details on smooth branched covers. In this local model, the pull-back 1-form $P^*(\a)$ is contact away from $\R^n\times\{0\}$ and it can be perturbed with compact support to a contact form near $\R^n\times\{0\}$. In particular, this contact local model can be glued to the global pull-back $p^*(\a)$, which is contact in the complement of $C\sse(Y,\ker\{\a\})$.

The precise result reads as follows.

\begin{thm}[\cite{Geiges97},\cite{OzturkNiederkruger07}]\label{bccontact}
Let $(Y,\ker\{\a\})$ be a contact manifold, $C\sse (Y,\ker\{\a\})$ a contact submanifold and $p: Y(C)\lr Y$ a smooth branched cover with branch locus $C$.

Then there exists a contact structure $\ker\{\a_1\}$ on $Y(C)$, unique up to contact isotopy, with a path $(\a_t)$ of $1$-forms such that

\begin{itemize}
	\item[-] $\alpha_0=p^*\alpha$, $\alpha_t$ is a contact 1-form for $t\in(0,1]$,
	\item[-] $d(\partial_t \alpha_t)|_{t=0}$ is a positive form on the normal bundle of the branched set $p^{-1}(C)$. 
\end{itemize}
\end{thm}

The contact manifold $(Y(C),\ker\{\a_1\})$ in Theorem \ref{bccontact} is called the contact branched cover of $(Y,\ker\{\a\})$ along the contact submanifold $C$.

\subsection{Open book decompositions}\label{ssec:ob}
In this article, we oftentimes describe contact manifolds by using adapted open book decompositions as introduced by E.~Giroux \cite{Giroux02}. We describe the basic features necessary for the present work, more details can be found in \cite{Colin08, vanKoert17} and \cite[Section 7.3]{Geiges08}.

Given a smooth manifold $W$ with boundary and a diffeomorphism compactly supported in the interior of $W$, $\varphi\in\Diff^c(W)$, we can consider its mapping torus
\[
T_\varphi = W\times [0,1]/\sim,
\]
where $(x,1)\sim(\varphi(x),0)$. The smooth manifold $T_\varphi$ has boundary $\partial  W\times S^1$, and we obtain a closed manifold $\ob(W,\varphi)$ by gluing $T_\varphi$ to $(\partial W)\times D^2$ with the identity diffeomorphism. The smooth manifold $\ob(W,\varphi)$ is said to have an open book decomposition associated to the pair $(W,\varphi)$ \cite{Rolfsen76}.
 
The main contact topology result we use is \cite[Proposition 9]{Giroux02}, a higher-dimensional version of the 3-dimensional construction of W.~Thurston and H.~Winkelnkemper, which states that a Liouville manifold $(W,\la)$ and a symplectomorphism compactly supported on the interior of $W$, $\varphi\in\Symp^c(W,d\la)$, determine a unique contact structure $\xi_\varphi$ on the smooth manifold $\ob(W,\varphi)$, up to contact isotopy. See \cite[Theorem 7.3.3]{Geiges08} for a proof.

Let $(Y,\xi)$ be a contact manifold. In the course of the article we will use the notation
$$(Y,\xi)=\ob(W,\varphi),$$
to indicate that the contact manifold $(Y,\xi)$ is contactomorphic to $(\ob(W,\varphi),\xi_\varphi)$, i.e. that $\xi$ is supported by the open book $\ob(W,\varphi)$ in the language of \cite{Geiges08}.

\begin{example}\label{ex:exampleOB}
Let $(T^*S^{n},\la_\std)$ be the standard Liouville structure on the cotangent bundle $T^*S^{n}$ of the standard round $n$-sphere $S^n$. The square of the distance to the zero section is a Lyapunov function for $X_{\la_\std}$, and thus endows $(T^*S^n, \la_\std)$ with a Weinstein structure \cite[Example 11.12.(2)]{CieliebakEliashberg12}.

Denote by $\tau_{S^{n}}\in\Symp^c(T^*S^{n},d\la_\std)$ the higher-dimensional Dehn twist along the zero section $S^{n}$. This is the compactly supported symplectomorphism initially introduced in \cite[Section 2]{Arnold95} in the 4-dimensional case. The general definition can be found in \cite[Chapter III.16.(c)]{Seidel08} and \cite[Section 4.3]{Colin08}, it is likely the best understood non-trivial symplectomorphism as of the writing of this article.

The Milnor fibration construction \cite{Giroux02} yields the contactomorphism
$$(S^{2n+1},\xi_{\st})=\ob((T^*S^{n},\la_\std),\tau_{S^{n}}).$$
In addition, the exact Lagrangian zero section $S^{n}\sse T^*S^{n}$ lifts to the standard Legendrian unknot $\La_0\sse (S^{2n+1},\xi_{\st})$, we refer to the reader to \cite{CasalsMurphyPresas} for details.\hfill$\Box$
\end{example}

\subsubsection{Open Book Stabilization}\label{sssec:obstab}

Let $(Y,\xi)=\ob(W,\p)$ be an adapted open book decomposition, the operation of stabilization of $\ob(W,\p)$ allows us to obtain a new open book also compatible with $(Y,\xi)$. This operation, which we now describe, was initially introduced in \cite{Giroux02}, and see also \cite[Section 4.3]{Colin08}.

Given $(Y,\xi)=\ob(W,\p)$, consider a properly embedded Lagrangian disk $D\sse(W,\la)$ with Legendrian boundary $\partial D\sse(\dd W,\ker\{\la|_{\dd W})$. Perform a critical Weinstein handle attachment \cite{CieliebakEliashberg12,Weinstein91} along the Legendrian sphere $\partial D\sse(\dd W,\ker\{\la|_{\dd W})$, resulting in a Weinstein manifold $(W_{\dd D},\la)$. The Lagrangian core of the Weinstein handle and the Lagrangian filling $D\sse (W_{\dd D},\la)$ glue along the Legendrian sphere $\dd D\sse(\dd W,\ker\{\la|_\dd W\})$ to form a Lagrangian sphere $S\sse (W_{\dd D},\la)$.

Let $\tau_S$ denote the Dehn twist along this exact Lagrangian sphere $S\sse(W_{\dd D},\la)$, and extend the compactly supported symplectomorphic $\varphi\in \Symp^c(W,d\la)$ by the identity to an eponymous compactly supported symplectomorphism $\varphi\in \Symp^c(W_{\dd D},d\la)$. The open book decomposition $(W_{\dd D},\tau_S\circ\p)$ of the contact manifold $\ob(W_{\dd D},\tau_S\circ\p)$ is said to be the stabilization of $(W,\p)$ along the Lagrangian disk $D\sse(W,\la)$.

The central property of this stabilization operation for contact open books is that it preserves the contactomorphic type of the resulting contact manifold. This is the content of the following:

\begin{thm}[\cite{Colin08,Giroux02,vanKoert17}]\label{thm:obstab}
The contact manifolds $(\ob(W,\varphi), \xi_\varphi)$ and $(\ob((W_{\dd D},\tau_S\circ\p)), \xi_{\tau_S\circ\p})$ are contactomorphic.
\end{thm}

Finally, the word stabilization appears in many related contexts in the contact topology literature. The open book stabilization described above is the geometric operation printed in the contact boundary when a stabilization of a symplectic Lefschetz fibration is performed \cite{CasalsMurphy}. (This is in turn equivalent to the stabilization of a Weinstein handlebody with a critical cancelling pair \cite[Section 2]{CasalsMurphy}.) Thus stabilizing symplectic Lefschetz fibrations and contact open books are equivalent. In addition, \cite[Section 6]{CasalsMurphyPresas} explains how the stabilization of a Legendrian submanifold introduced in Subsection~\ref{ssec:leg} is related to the negative stabilization of a contact open book, and see also \cite{Etnyre04b} for its 3-dimensional precedent.

\begin{example}
Consider the contact open book decomposition $\ob(T^*S^{n},\tau_{S^{n}})$ in Example~\ref{ex:exampleOB}. The cotangent fiber at any point is a Lagrangian disk $D\sse(T^*S^{n},\la_{\std})$, and we can perform a stabilization of $\ob(T^*S^{n},\tau_{S^{n}})$ along this Lagrangian disk. The Weinstein manifold obtained by the Weinstein handle attachment along $\dd D\sse\dd(T^*S^{n},\la_{\std})$ is the plumbing of two copies of $(T^*S^n,\la_\st)$, and the resulting monodromy is the composition of two Dehn twists, one along each of the two spherical zero sections.

This stabilization can be performed iteratively by choosing the Lagrangian disk $D$ to be a cotangent fiber of the last copy of $(T^*S^{n},\la_{\std})$ being plumbed, where the cotangent fiber is chosen to be disjoint from the intersection created in the plumbing procedure \cite{CasalsMurphy,vanKoert17}.

\begin{figure}[ht]{\tiny
		\begin{overpic}%[grid,tics=10] 
			{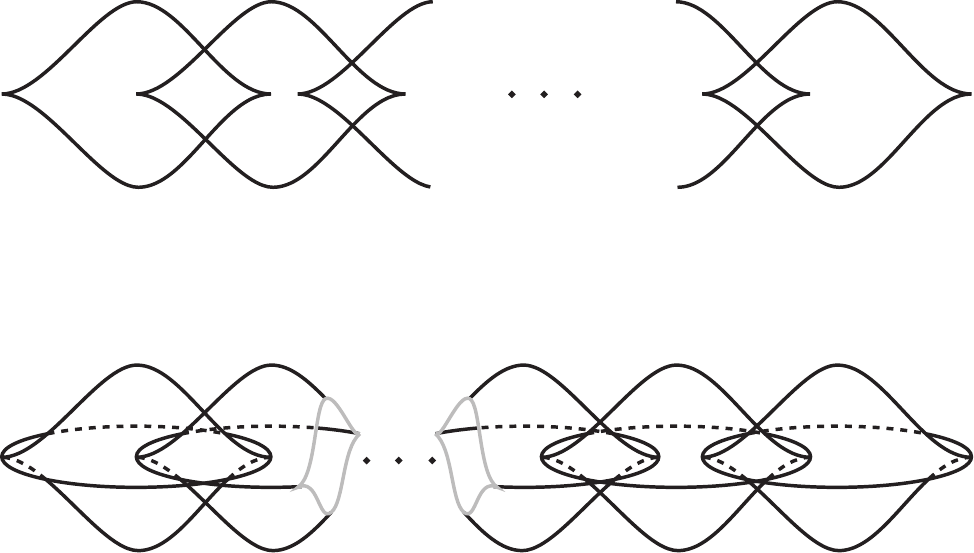}
			%\put(120, 120){$+$}
	\end{overpic}}
	\caption{Weinstein handle diagrams for the 4- and 6-dimensional Milnor fibers of the $A_k$ singularity, where there are $k$ Legendrian spheres in each figure, corresponding to the Milnor number $\mu(A_k)=k$.}
	\label{AkFiber}
\end{figure}

In precise detail, let $(A^{2n}_k,\la_\st)$ be the symplectic plumbing of $k$ copies of $(T^*S^{n},\la_{\std})$ according to the $A_k$-Dynkin diagram  \cite[Example 19.3(iii)]{Seidel08}, and especially \cite[Section III.20]{Seidel08}. The $k^{th}$ stabilization of the contact open book $\ob(T^*S^{n},\tau_{S^{n}})$ as described above yields the contact open book decomposition

$$(S^{2n+1},\xi_{\st})=\ob(A^{2n}_k,\p_k),\quad\forall k\in\N,$$

where $\p_k$ is the composition of one Dehn twist along each of the zero sections. Figure~\ref{AkFiber} depicts a Legendrian handlebody presentation for the Weinstein manifolds $A^{2n}_k$ \cite{CasalsMurphy}.

These contact open book decompositions correspond to the symplectic Milnor fibration for the $A_k$-isolated singularities
$$f_k:\C^{n+2}\lr\C: (z_1,\ldots,z_n)\lr z_1^{k+1}+z_2^2+\ldots+z_n^2,$$
the Weinstein manifold $A^{2n}_k$ being the Milnor fiber of the holomorphic map $f_k$ \cite{Arnold95,Giroux02}.\hfill$\Box$
\end{example}

% ------------------------------------------------------------------------------------------------------
% ------------------------------------------------------------------------------------------------------
\section{Contact push-offs and formal isotopy classes}\label{sec:pushoff}

Transverse push-offs of 1-dimensional Legendrian knots $\La\sse(S^3,\xi_\st)$, \cite[Section 2.9]{Etnyre05} and \cite[Section 3.1]{Geiges08}, are central in the study of transverse and Legendrian knot invariants \cite{HarveyKawamuroPlamenevskaya09,OzsvathSzaboThurston08}. In this section we introduce the notion of a contact push-off of a Legendrian submanifold in all higher-dimensions, which gives rise to a wealth of codimension-2 isocontact embeddings in a given contact manifold.

% ------------------------------------------------------------------------------------------------------
\subsection{The contact push-off of Legendrian submanifolds}\label{ssec:contactpushoff}

Let $(Y,\xi)$ be a contact manifold and $\La$ a smooth Legendrian submanifold. By the Weinstein neighbourhood theorem \cite[Proposition 2.5.5]{Geiges08} there exists an open set $\Op(\La)\sse(Y,\xi)$ containing $\La$ and a contactomorphism
$$(\Op(\La),\xi|_{\Op(\La)})\cong(J^1\La,\xi_\std),$$
which identifies $\La\sse\Op(\La)$ with the zero section $\La\times\{0\}$ in the 1-jet space $J^1\La=T^*\La\times\R$. Here $\xi_\std$ on $J^1\La$ is given as the kernel of $dz-\lambda_\st$ where $z$ is the coordinate on $\R$ and $\lambda_\st$ is the canonical 1-form on the cotangent bundle $T^*\La$. 

Let $D^*(\La)$ denote the unit disk bundle in the cotangent bundle of $\La$, for an arbitrary fixed Riemannian metric on $\La$. The hypersurface $(D^*(\La),\la_\st)\sse (J^1(\La),\xi_\st)$ is a Weinstein hypersurface of the contact manifold $J^1(\La)$, as defined in Subsection \ref{ssec:WeinsteinHyper}, and thus its boundary $(\dd(D^*\La),\la_\st|_{\dd(D^*\La)})$, is a contact submanifold of $(J^1(\La),\xi_\st)$. We denote this codimension--2 contact submanifold by 
$$\tau(\La):=(\dd(D^*\La)\times\{0\},\la_\st|_{\dd(D^*\La)})\sse (J^1\La,\xi_\std),$$
and the Weinstein hypersurface $D^*\La$ will be referred to as a Weinstein neighborhood of the Legendrian $\La\sse(J^1(\La),\xi_\st)$. In order to ease notation, we will denote the disk bundle $D^*(\La)$ and its boundary $\dd D^*(\La)$, for any choice of metric, by $T^*(\La)$ and $\dd T^*(\La)$.

\begin{definition}\label{def:pushoff}
Let $(Y,\xi)$ be a contact manifold and $\La$ a Legendrian submanifold. The contact push-off of $\La$ in $(Y,\xi)$ is the image of the contact submanifold
$$\tau(\La)=(\dd(T^*\La)\times\{0\},\xi_\std)\sse (J^1\La,\xi_\std),$$
by a contactomorphism identifying a neighborhood of $\La$ with $J^1\La$:
$$(\Op(\La),\xi|_{\Op(\La)})\cong(J^1\La,\xi_\std).$$
The contact push-off is eponymously referred to as $\tau(\La)$.\hfill$\Box$
\end{definition}

Definition \ref{def:pushoff} in the 3-dimensional case asserts that the contact pushoff of a Legendrian knot is defined as the transverse link given by the disjoint union of the positive and the negative transverse push-offs \cite[Section 3.1]{Geiges08}.

\begin{remark}
The positive and negative transverse push-off of a 1-dimensional Legendrian knot $\La\sse(\R^3,\xi_\st)$ do not independently determine the Legendrian isotopy type of $\La$ \cite[Theorem 2.1]{EpsteinFuchsMeyer01}.\hfill$\Box$
\end{remark}

In line with this remark, we would like to include the following conjecture, which aims at underscoring the role of contact push-offs for higher-dimensional Legendrian topology:

\begin{conj}\label{conj:ribbon}
	Let $\La_0,\La_1\sse(S^{2n+1},\xi_\st)$ be two formally Legendrian isotopic smooth Legendrian submanifolds such that $\tau(\La_0)$ and $\tau(\La_1)$ are contact isotopic. Then $\La_0$ and $\La_1$ are Legendrian isotopic.\hfill$\Box$
\end{conj}

% ------------------------------------------------------------------------------------------------------
\subsection{Formal isotopy type} Let us address the formal contact isotopy type of contact push-offs. First, the formal contact structure of the contact push-off $\tau(\La)$ is induced by the almost complex, equivalently symplectic, structure of the Liouville hypersurface $(\La,\omega)$. Given a Legendrian isotopy $\{\La_t\}$, $t\in[0,1]$, with $\La_0=\La$, the almost complex structures on $(T^*\La_t,\omega_t)$ would produce in their boundaries a (formal) contact isotopy between $\tau(\La_0)$ and $\tau(\La_1)$.

In this article we are comparing Legendrians with respect to formal Legendrian isotopy, and thus we need the effect of a formal Legendrian isotopy on the formal contact isotopy type of the contact push-offs. The main technical lemma reads as follows: 

\begin{lemma}\label{lem:formaldetermined}
Let $\La_0,\La_1\sse(Y,\xi)$ be two Legendrian submanifolds. If $\Lambda_0$ and $\Lambda_1$ are formally Legendrian isotopic, then $\tau(\La_0)$ and $\tau(\La_1)$ are formally contact isotopic as isocontact embeddings. 
\end{lemma}

The most direct proof for Lemma \ref{lem:formaldetermined} is to define the formal contact push-off of a formal Legendrian $(f,\{F_s\})$, extending Definition \ref{def:pushoff} to formal Legendrian embeddings. Lemma \ref{lem:formaldetermined} then follows from the fact that the formal Legendrian type of $(f,F_s)$ determines the formal contact type of $\tau((f,F_s))$. Let $(Y,\xi)$ be a $(2n+1)$-dimensional contact manifold with a choice of contact form $\alpha$ and a Riemannian metric $g$. Let $(\La,f,F_s)$, $s\in[0,1]$, be a formal Legendrian embedding into $(Y,\xi,\tau)$, we construct the formal contact push-off as follows.

Consider a family of automorphisms $\{G_s\}_{s\in[0,1]}$, $G_s:TY\lr TY$, such that $G_s\circ F_0=F_s$ as a family of monomorphisms $T\La\lr TY$. Consider the formal contact structures $(\xi_s,\tau_s)_{s\in[0,1]}$ defined by $(\xi_1,\tau_1)=(\xi,d\alpha)$, $\xi_s=(G_1^{-1}\circ G_s)(\xi_1)$ and $\tau_s=\tau_1\circ G_1\circ G_s^{-1}$. Then the image $F_0(T\La)$ is now contained in $\xi_0$, and it is a Lagrangian subbundle with respect to the conformally symplectic structure $(\xi_0,\tau_0)$.

The formal contact push-off $\tau(\La,f,F_s)$ of the formal Legendrian embedding $(\La,f,F_s)$ is defined to be the unit normal bundle of $F_0(T\La)\sse(\xi_0,\tau_0)$ endowed with the formal contact structure $(\Xi,\tau|_\Xi)$, where $\Xi=\langle\nu_1,\nu_2\rangle^{\perp_{\tau_0}}$, with $\nu_1\in\Gamma(\xi_0)$ a vector transverse to $\tau(\La,f,F_s)$ and $\nu_2\in \Gamma(T(\tau(\La,f,F_s)))$ is such that $\langle\nu_1,\nu_2\rangle$ is a symplectic subspace of $(\xi_0,\tau_0)$. By construction, this yields a formal isocontact embedding of $\tau(\La,f,F_s)$ into the formal contact manifold $(Y,\xi_0,\tau_0)$. Note that the $s$-parametric nature of the construction, the formal isocontact embedding $j_0:\tau(\La,f,F_s)\lr(Y,\xi_0,\tau_0)$ extends to a family of formal isocontact embeddings $j_s:\tau(\La,f,F_s)\lr(Y,\xi_s,\tau_s)$, $s\in[0,1]$.

\begin{proof}[Proof of Lemma \ref{lem:formaldetermined}] Consider a formal Legendrian isotopy $(f_t,F_{s,t})$, $s,t\in[0,1]$, between $\La_0$ and $\La_1$. The formal contact pushoffs $\tau(f_t,F_{s,t})$ yield a path of formal contact submanifolds into the formal contact manifold $(Y,\xi)$ which interpolate ambiently between the formal contact type of $\tau(\La_0)$ and the formal contact type of $\tau(\La_1)$, thus proving that $\tau(\La_0)$ and $\tau(\La_1)$ are formally contact isotopic isocontact embeddings.
\end{proof}

% ------------------------------------------------------------------------------------------------------
\subsection{Singular Legendrians} In this section, we use singular Legendrian submanifolds to construct further isocontact embeddings of codimension-2 smooth contact manifolds. Such isocontact embeddings are obtained by generalizing Definition \ref{def:pushoff} to a class of singular Legendrians $\La\sse(Y,\xi)$. This allows us to prove Theorem~\ref{cor:infinitelymany} and emphasize the generality in which our methods apply. The reader interested in the proof of Theorem \ref{thm:main} can restrict themselves to smooth Legendrians and proceed to Section \ref{sec:covers} on the first reading.

\begin{remark} Recent advances in symplectic topology \cite{Fukaya02,Nadler17,Starkston}, in combination with the compactness theory of integral currents \cite{FedererFleming60,White89}, strongly indicate that there should exist a Floer theory with singular boundary conditions. In particular, singular Legendrian submanifolds in contact topology should potentially be on equal footing with smooth Legendrians. This article hopefully begins to illustrate this.\hfill$\Box$
\end{remark}

Let $(W,\la,\p)$ be a $2n$-dimensional Weinstein manifold \cite[Section 1.12]{EliashbergGromov91} and $v$ its Liouville vector field, defined by the equation $\iota_v d\lambda = \lambda$. Consider the time-$t$ flow $\phi_t:W\lr W$ of the Liouville vector field, and its skeleton $\Sk(W)=\Sk(W,\la)\sse (W,\la)$, defined as the intersection 
$$\Sk(W) = \cap_{t\in (0, \infty)} \phi_{-t}(W).$$
The existence of the adapted plurisubharmonic function $\p\in C^\infty(W)$ shows that the skeleton $\Sk(W)$ is an isotropic complex \cite[Chapter 11]{CieliebakEliashberg12}, but it can be readily thickened to a Lagrangian CW complex \cite[Section 2.1]{Starkston}, and see also \cite[Section 1]{Eliashberg18}.

\begin{definition}\label{def:generalcontactpushoff}
Let $(Y,\xi)$ be a $(2n+1)$-dimensional contact manifold and $(W,\la,\p)\sse(Y,\xi)$ a Weinstein hypersurface. The contact submanifold $\tau(W)=\dd(W,\la,\p)\sse(Y,\xi)$ will be called the contact push-off of the Legendrian lift $\La(W)\sse(Y,\xi)$ of the exact Lagrangian skeleton $\Sk(W)\sse(W,\la,\p)$. The contact manifold $\tau(W)$ will also be referred to as the contact push-off of $(W,\la)$.\hfill$\Box$
\end{definition}

In Subsection \ref{ssec:contactpushoff} we have discussed the case in which $(W,\la,\p)\cong\Op(\La)$ where $\La\sse(Y,\xi)$ is a smooth Legendrian submanifold. Then Definition \ref{def:generalcontactpushoff} yields a smooth Lagrangian skeleton and recovers Definition \ref{def:pushoff}.

\begin{remark}
The emphasis on the skeleton $\La(W)$ in Definition \ref{def:generalcontactpushoff}, instead of just considering $(W,\la)$, is meant to align Definitions \ref{def:pushoff} and \ref{def:generalcontactpushoff} and stress the connection between Legendrian submanifolds and isocontact embeddings.
\end{remark}

Instead of generalizing Lemma \ref{lem:formaldetermined}, we shall restrict ourselves to the case where the lemma can be directly applied, as follows. Let $\La\sse(Y,\xi)$ a singular Legendrian submanifold, following Definition \ref{def:legstab} of a stabilization of a smooth Legendrian \cite{EkholmEtnyreSullivan05a,Murphy??}, we define the Legendrian stabilization of $\La$ as the Legendrian submanifold $s(\La)\sse (Y,\xi)$ obtained by performing one stabilization along each of the smooth top-dimensional strata of $\La$. This is well-defined since stabilization is a local operation. The $h$-principle in \cite{Murphy??} can be adapted to this singular setting and the result of stabilizing a singular Legendrian yields the unique loose representative in the formal Legendrian isotopy class.

\begin{lemma}\label{lem:singpushoff}
Let $(Y,\xi)$ be a contact manifold, $(W,\la)\sse(Y,\xi)$ a Weinstein hypersurface and $p\in\Sk(W,\la)$ a smooth point. There exists a Weinstein hypersurface $(W_p,\la_p)\sse(Y,\xi)$, homotopic to $(W,\la)$ as a formal symplectic hypersurface, such that the Legendrian lift of the skeleton $\Sk(W_p,\la_p)$ is obtained from $\Sk(W,\la)$ by a stabilization of the Legendrian lift of $\Sk(W,\la)$ along the lift of the smooth point $p$.
\end{lemma}

\begin{proof}
The Weinstein hypersurface $(W_p,\la_p)\sse(Y,\xi)$ is constructed in two pieces. The first piece is the complement $W\setminus A(p)$ of the stable manifold, i.e. the complement the basin of attraction $A(p)$ of a small smooth neighborhood $\Op\{p\}\sse\Sk(W,\la)$, for the Liouville dynamical system defined by the Liouville field $X_\la$. Thus $(W_p,\la_p)$ is defined to coincide with $W\setminus A(p)$, in the domain of definition of the later. The second piece is the Weinstein neighborhood of the Legendrian connected sum of the Legendrian lift of $\Op(p)\sse\Sk(W,\la)$ with a (disjoint) stabilized Legendrian unknot $s(\La_0)$ embedded in a contact Darboux chart around $p\in(Y,\xi)$. Since both pieces coincide along the germs of their boundaries, they glue together to define $(W_p,\la_p)$ \cite{Eliashberg18}.

Since the stabilization of a Legendrian is formally Legendrian homotopic relative to the boundary of a small contact neighborhood \cite{EkholmEtnyreSullivan05a}, the argument in Lemma \ref{lem:formaldetermined} above shows that the Weinstein neighborhood of the stabilized Legendrian is formally homotopic, relative to the boundary and as a symplectic hypersurface, to that of the non-stabilized Legendrian. This implies that $(W_p,\la_p)$ is formally homotopic \cite{EliashbergMishachev02}, as a symplectic hypersurface, to $(W,\la)$, even relative to their common intersection in the complement $W\setminus A(p)$ of the region being stabilized.
\end{proof}

In Definition \ref{def:generalcontactpushoff} we introduced the notion of the contact push-off of a singular Legendrian, where in this singular case we always assume that the Legendrian is being presented as the lift of the Lagrangian skeleton of a given Weinstein hypersurface. In this situation, Lemma \ref{lem:singpushoff} allows us to talk about the contact push-off of the stabilization of a given singular Legendrian. This construction will be used in the proof of Corollary \ref{conj:infinite}. Let us end this section with an example.

The theory of arboreal singularities \cite{Nadler17,Starkston} provides many interesting examples of singular Legendrians $\La\sse(\R^{2n+1},\xi_\st)$, arising as the Legendrian boundaries of singular Lagrangian arboreal skeleta \cite[Section 2.4]{Starkston}. Arboreal singularities can more directly be described by using the following general class of singular Legendrians.

\begin{example}
Let $f:\C^{n+1}\lr\C$ be a holomorphic polynomial and let
$$(S^{2n+1},\xi_\st)=\ob((W_f,\la_f),\tau_f)$$
be the open book associated to the Milnor fibration of $f$ \cite[Section A]{Giroux02}. The set-theoretical union $L_f\sse(W_f,\la_f)$ of the Lagrangian vanishing cycles \cite{Seidel08} constitute an exact Lagrangian skeleton of the Weinstein page $(W_f,\la_f)$, and thus the Legendrian lift $\La(L_f)\sse (S^{2n+1},\xi_\st)$ defines a singular Legendrian. In this case $\tau(L_f)$ coincides with the contact binding
$$(\dd W_f,\la_f|_{\dd W_f})$$
of the adapted contact open book \cite[Section 4]{Colin08}. For instance, a simple class of singular Legendrians is obtained in this manner by considering stabilizations of plane curve singularities $f:\C\lr\C$ to $f:\C^{n+1}\lr\C$, in which case the closures of the smooth strata of the Legendrian skeleta $\La(L_f)$ are $S^n$-spheres, and these smooth strata meet only along ordinary double points, thus topologically forming a plumbing graph of $n$-dimensional spheres.\hfill$\Box$
\end{example}

\section{Contact surgery presentations of contact cyclic branched covers}\label{sec:covers}

Let $(Y,\xi)$ be a contact manifold, $\La\sse(Y,\xi)$ a Legendrian submanifold, possibly singular, and $\tau(\La)$ its smooth contact push-off. 
The contact type of the contact branch covers of $\La\sse(Y,\xi)$ along $\tau(\La)$ is the invariant used to prove Theorem~\ref{thm:main} and Theorem~\ref{cor:infinitelymany}.

The main goal of this section is to construct a contact surgery presentation of the contact $n$-fold cyclic branched covers of $(Y,\xi)$ along $\tau(\La)$, $(C_n(\tau(\La)),\xi_n(\tau(\La)))$. In the $3$-dimensional case, the articles \cite{HarveyKawamuroPlamenevskaya09,Plamenevskaya06a} discuss contact surgery presentations for contact branched covers along transverse knots, and the unpublished work \cite[Section 7]{Avdek} contains part of the ideas we develop in this section.

The techniques developed in Subsection \ref{ssec:contactcyclicsurgery} below provide a contact surgery presentation for $$(C_n(\tau(\La)),\xi_n(\tau(\La)))$$
in all dimensions and for all $n\in\N$. In fact, any $n\geq2$ can be used to prove Theorems~\ref{thm:main} and~\ref{conj:infinite}. For simplicity, we will only discuss the applications in the case of the cyclic $2$-fold branched covers. The main result from this section, proven in Subsection~\ref{ssec:proofmain} and used in the proof of Theorem~\ref{thm:main}, is the following result.

\begin{thm}\label{thm:contactsurgery}
Let $\La\sse(S^{2n+1},\xi_\std)$ be a Legendrian sphere and $\tau(\La)$ its contact push-off. The contact double branched cover of $(S^{2n+1},\xi_\std)$ along $\tau(\La)$ admits a Weinstein filling obtained by attaching a Weinstein $(n+1)$--handle to $(D^{2n+2},\la_\st)$ along the Legendrian connected sum $\La\#\La$ in $(S^{2n+1},\xi_\std)$. 
\end{thm}

Theorem~\ref{thm:contactsurgery} also can be interpreted as a statement about the existence of a contact surgery presentation of the contact double branched cover; more specifically the contact double branched cover is obtained by ($-1$)-Legendrian surgery on the Legendrian connected sum $\La\#\La$. %In this case, the contact surgery diagram is all of $(-1)$--surgeries and thus we can translate to the language of Weinstein fillings \cite{CasalsMurphy,CieliebakEliashberg12}.

The statement of Theorem~\ref{thm:contactsurgery} implicitly identifies $(S^{2n+1},\xi_\std)$ as the contact boundary of $(D^{2n+2},\la_\st)$.  Theorem~\ref{thm:contactsurgery} is optimally absorbed with a visual diagram, as in the following example.

\begin{example}\label{ex:Z0contactsurgery}

Consider a Legendrian sphere $\La\sse(S^{2n+1},\xi_\std)$, we can assume after higher-dimensional Reidemeister moves \cite[Section 2]{CasalsMurphy} that its Legendrian front diagram has the form shown in the upper left corner of Figure~\ref{fig:Z0ContactSurgery}. Theorem \ref{thm:contactsurgery} states that the contact double branched cover of $(S^{2n+1},\xi_\std)$ along $\tau(\La)$ bounds the Weinstein manifold obtained by attaching a Weinstein $(n+1)$--handle to the Legendrian sphere in the upper right of Figure~\ref{fig:Z0ContactSurgery}

\begin{figure}[ht]{\tiny
\begin{overpic}%[grid,tics=10] 
{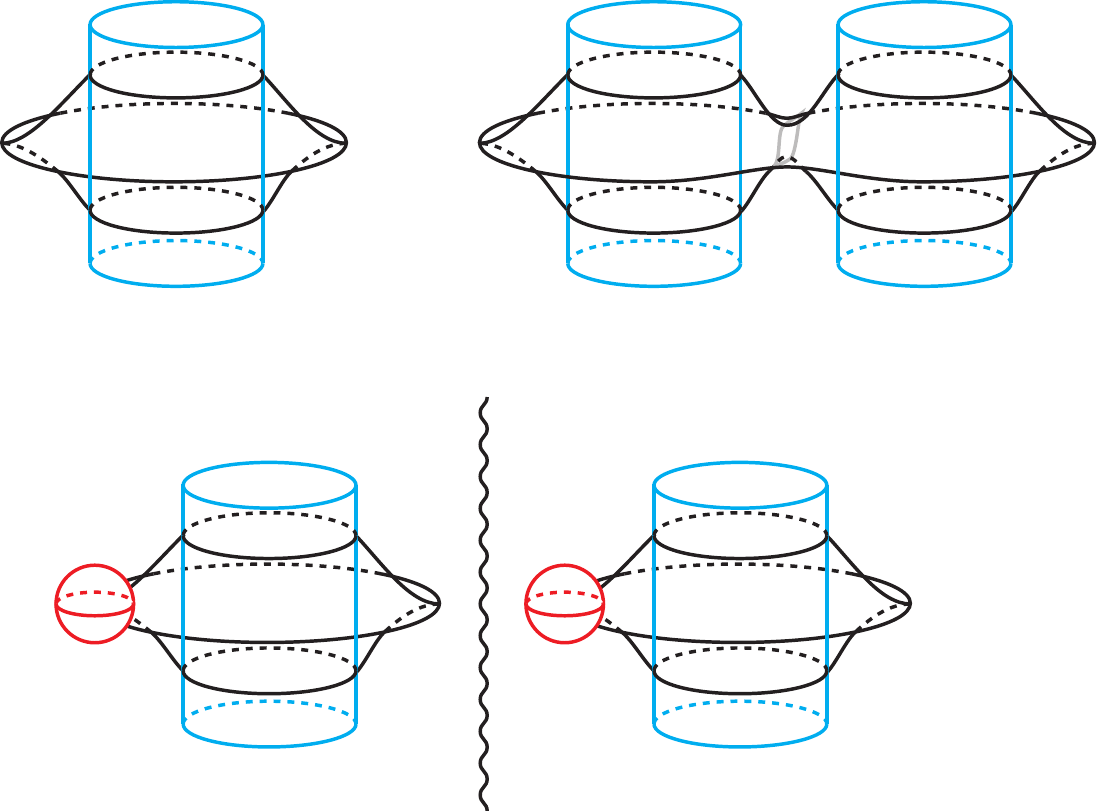}
%\put(120, 120){$+$}
\end{overpic}}
\caption{The Legendrian knot $\La$ is depicted in the upper left corner, where the blue cylinder is any front diagram for a Legendrian tangle. The upper right corner is a Legendrian handlebody presentation of a Weinstein filling for the contact double branched cover of $(S^{2n+1},\xi_\std)$ along the contact push-off $\tau(\La)$. The lower row depicts the same Weinstein filling with the natural Weinstein structure coming from our construction in Subsection~\ref{ssec:contactcyclicsurgery}. The lower row simplifies to the upper right diagram via canceling Weinstein handles, as we explain in the proof of Theorem~\ref{thm:contactsurgery}. In short, the blue cylinder represents a Legendrian tangle, the red sphere indicates the index-1 surgery and the wave vertical segment in the lower row indicates two disjoint regions for Weinstein front diagrams.}
\label{fig:Z0ContactSurgery}
\end{figure}

For instance, in the case of the Legendrian unknot $\La=\La_0$ we obtain that the double branched cover of $(S^{2n+1},\xi_\std)$ along the contact submanifold $(\tau(\La_0),\xi_{\st})=(\dd(T^*S^n),\xi_{\st})$ is $(\dd(T^*S^{n+1}),\xi_{\st})$. Indeed, this also can be proven by using that the algebraic double branched cover of $\C^{n+1}$ along the affine conic hypersurface $Q^n\sse \C^{n+1}$ is affine isomorphic to the affine conic $Q^{n+1}\sse \C^{n+2}$, and restricting to the contact boundaries at infinity \cite{CasalsMurphy}.\hfill$\Box$
\end{example}

The techniques developed in this section apply, in particular, to Weinstein handlebody diagrams and our statements on contact branched covers along the contact boundaries of Weinstein hypersurfaces \cite{CieliebakEliashberg12,Eliashberg18} can be readily used to obtain new results for Weinstein manifolds, such as the following.

\begin{cor}\label{cor:flexiblebranch}
Let $(W,\la)$ be a flexible Weinstein manifold, $\La\sse(\dd W,\la|_{\dd W})$ a loose Legendrian submanifold and $(R(\La),\la_\st)\sse (W,\la)$ its Weinstein neighborhood. The contact cyclic branched cover of $\partial W$ along $\partial R(\La)$ has a flexible Weinstein filling. \hfill$\Box$
\end{cor}
%\begin{remark}
%In fact, there is a notion of Weinstein cyclic branched of $(W,\la)$ along $(R(\La),\la_\st)$ pushed into the interior of $W$ and one may show that this brach cover will be the flexible Weinstein manifold in the corollary
%\end{remark}

The proof of Theorem \ref{thm:contactsurgery} is given in Subsection \ref{ssec:proofmain} as a consequence of the mechanism that describes a Weinstein handle presentation for a cobordism from copies of a contact manifold to its cyclic branched covers, which we now introduce in the next subsection.

% ------------------------------------------------------------------------------------------------------
\subsection{Contact cyclic branched covers}\label{ssec:contactcyclicsurgery}

Let us start in the smooth category, where $M$ is a smooth manifold and $\Sigma\sse M$ is an embedded hypersurface with boundary $B=\dd\Sigma$. Consider the morphism
$$\iota_\Sigma:\pi_1(M\setminus B)\lr\Z: [\gamma]\longmapsto |\gamma\cap\Sigma|,$$
where $\gamma\sse M\setminus B$ is a representative intersecting only the interior of $\Sigma$ and such that the intersections are transverse. The kernel of the composition $p_n\circ\iota_\Sigma$, where $p_n:\Z\to\Z_n$ is reduction modulo $n$, defines the subgroup $\ker{(p_n\circ\iota_\Sigma)}\sse \pi_1(M\setminus B)$ and thus a covering space of $M\setminus B$. The completion of this covering space to a branched cover $C_n(B)$ over the manifold $M$ is called the $n$--fold cyclic branched cover of $M$ along the submanifold $B$.

In Subsection~\ref{ssec:contactbranch} we discussed the basics of contact branched covers. If $(Y,\xi)$ is a contact manifold, $B\sse (Y,\xi)$ a null-homologous contact submanifold, then the cyclic $n$--fold contact branched constructed in Theorem~\ref{bccontact} is denoted by $(C_n(B),\xi_n(B))$. 

In line with the smooth setting, a Weinstein hypersurface $(W,\la)\sse(Y,\xi)$ bounding $(B,\xi_B)$ can be used to construct $(C_n(B),\xi_n(B))$, as we now describe. We shall thus set $B=\dd W$ onwards in this section. Let $(W,\la)$ be a Weinstein domain and consider the stabilized Weinstein domain $(W\times D^2,\la+\la_\st)$. The projection onto the second factor restricted to the contact boundary $(\dd(W\times D^2),\ker\{\la+\la_\st\})$ yields the adapted contact open book decomposition
$$\dd((W\times D^2),\ker\{\la+\la_\st\})=\ob((W,\la),\id),$$
with page $(W,\la)$ and monodromy the identity $\id\in\Symp^c(W,\la)$. Let $(W_\theta,\la)\sse\ob((W,\la),\id)$ denote the interior of the page at angle $\theta\in S^1$, such that $W_\theta\cap W_{\theta'}=\emptyset$ for $\theta\neq\theta'$. Let $W_\theta'$ be the complement $\overline{W}_\theta\setminus\Op(\dd\overline{W}_\theta)$, i.e. the open subset of $W_\theta$ obtained by removing from $W_\theta$ its intersection with a small collar neighborhood of the binding. Then the contact cyclic branched cover with $(W,\la)$ as a Weinstein (Seifert) hypersurface is described as follows.

\begin{prop}\label{prop:cycliccoverasWeinsteinsum}
Let $(Y,\xi)$ be a contact manifold and $(W,\la)\sse(Y,\xi)$ a Weinstein hypersurface. The contact manifold $(C_n(\dd W),\xi_n(\dd W))$ is contactomorphic to the Weinstein sum of the contact manifold $\ob((W,\la),\id)$ with $n$-disjoint copies of $(Y,\xi)$ where the $j^{th}$ copy of $(Y,\xi)$ and $\ob((W,\la),\id)$ are glued along a page in $Y$ and $W'_{2\pi j/n}\sse\ob((W,\la),\id)$.
\end{prop}

\begin{proof}
Consider the symplectic $\Z_n$-action on $(W\times D^2,\la+\la_\st)$ given by
$$\Z_n\times (W\times D^2,\la+\la_\st)\lr (W\times D^2,\la+\la_\st): (e^{2\pi ik/n};w,z)\longmapsto(w,e^{2\pi ik/n}z),$$
where $D^2$ is the unit disk in $(\C,\la_\st)$. This action induces a contact $\Z_n$-action on the contact boundary $\ob((W,\la),\id)$, where $e^{2\pi ik/n}\in \Z_n$ sends the Weinstein hypersurface $W_{2\pi j/n}\sse\ob((W,\la),\id)$ to the hypersurface $W_{2\pi(j+k)/n}$, the index understood modulo $2\pi$. The fixed point set of this contact $\Z_n$-action is the contact binding $(\dd W,\xi)\sse\ob((W,\la),\id)$, and the quotient contact manifold is contactomorphic to $\ob((W,\la),\id)$. This expresses $\ob((W,\la),\id)$ as the contact $n$-fold cyclic branch cover of $\ob((W,\la),\id)$ along the contact binding $(\dd W,\xi_{\dd W})$, the crucial fact being that the $n$th power of the monodromy, being the identity, is (symplectically isotopic to) itself.

In this framework, perform the Weinstein sum of the contact manifold $\ob((W,\la),\id)$ with $n$ disjoint copies of $(Y,\xi)$ glued along the $n$ disjoint Weinstein hypersurfaces
$$W'_{2\pi k/n}\sse\ob((W,\la),\id),$$
for $1\leq k\leq n$, and extend the above $\Z_n$-action on $\ob((W,\la),\id)$ to this contact manifold. This is done in a natural manner, an element $e^{2\pi ik/n}\in\Z/n\Z$ contactomorphically sends the copy of $(Y,\xi)$ glued along $W'_{2\pi i/n}$ to the copy of $(Y,\xi)$ glued along $W'_{2\pi (i+k)/n}$. The quotient of this contact $\Z/n\Z$ is now contactomorphic to the Weinstein sum of $\ob((W,\la),\id)$ and $(Y,\xi)$ along the Weinstein surface $(W,\la)$, where $(W,\la)$ is any arbitrary page in $\ob((W,\la),\id)$. The fixed point set is still the contact binding in $\ob((W,\la),\id)$.

The above paragraph expresses the Weinstein sum in the statement as the contact cyclic $n$--fold branch cover of the Weinstein sum
$$\ob((W,\la),\id){}_{(W,\la)}{\#}{_{(W,\la)}} (Y,\xi),$$
with branch locus $\dd W\sse \ob((W,\la),\id)$. Hence, in order to conclude Proposition~\ref{prop:cycliccoverasWeinsteinsum} it suffices to notice that this Weinstein sum is contactomorphic to $(Y,\xi)$. Indeed, by the definition of an adapted open book \cite[Section 4]{Colin08} the complement of the closure of a page $\overline{W}\sse\ob((W,\la),\id)$, and thus $(\ob((W,\la),\id)\setminus(W,\la),\xi_\st)$, is contactomorphic to the contactization
$$(W\times [0,1],\ker\{\la+ds\})$$
of the Weinstein page $(W,\la)$, which is the standard neighborhood of the Weinstein hypersurface $(W,\la)$ discussed in Subsection \ref{ssec:WeinsteinHyper}. Thus the Weinstein sum above removes the interior of a standard neighborhood of $(W,\la)$ in $(Y,\xi)$ and replaces it with a contactomorphic copy of itself glued with the identity along the boundary. This concludes the required statement.
\end{proof}

Proposition~\ref{prop:cycliccoverasWeinsteinsum} gives the following direct description of the cyclic $n$--fold contact branched cover $(C_n(\partial W),\xi_n(\partial W))$ of $(Y,\xi)$ along the boundary of the Weinstein hypersurface $(W,\la)$. In a standard neighborhood $(W\times [-\e, \e],\ker{\la-ds})\sse(Y,\xi)$, for $\e\in\R^+$, consider the Weinstein hypersurfaces $W_k=W\times\{-\delta+2k\delta/n\}$ for $1\leq k\leq n$, where $0<\delta\leq\e$ is arbitrary but fixed.

\begin{cor}\label{cor:neededdescription}
Let $(Y,\xi)$ be a contact manifold and $(W,\la)\sse(Y,\xi)$ a Weinstein hypersurface. The contact manifold $(C_n(\partial W),\xi_n(\partial W))$ is contactomorphic to the Weinstein connected sum of $(Y,\xi)$ with $(n-1)$ copies $(Y_k,\xi)$, $1\leq k\leq (n-1)$, where the original $(Y,\xi)$ is summed to $(Y_k,\xi)$ along $(W_k,\la)\sse (Y,\xi)$ and $(W,\la)\sse(Y_k,\xi)$.
\hfill$\Box$
\end{cor}

In particular, Corollary \ref{cor:neededdescription} recovers \cite[Theorem 1.20]{Avdek}. The description of $(C_n(\partial W),\xi_n(\partial W))$ in Corollary \ref{cor:neededdescription} follows from the proof of Proposition~\ref{prop:cycliccoverasWeinsteinsum} since after Weinstein summing $(Y,\xi)$ to $\ob((W,\la),\id)$ along $(W,\la)$ and $(W_1,\la)$,  the $W_{2\pi k/n}$ for $2\leq k\leq n$ in  $\ob((W,\la),\id)$ become the $(k-1)$ copy of $W\sse(Y,\xi)$ under the identification of $(Y,\xi)$ with the Weinstein sum of $(Y,\xi)$ and $\ob((W,\la),\id)$. Note that we have included Corollary \ref{cor:neededdescription} for all $n\in\N$ for completeness, the case $n=2$ will suffice to prove Theorem \ref{thm:main}.

%----------------------------------------
\subsection{Weinstein fillings}\label{ssec:proofmain} 
Let us establish the Legendrian handlebody presentation of the Weinstein filling of the cyclic $2$--fold branched cover in Theorem~\ref{thm:contactsurgery}.

\begin{proof}[Proof of Theorem~\ref{thm:contactsurgery}]
First, let us apply Proposition \ref{prop:cycliccoverasWeinsteinsum} to the Weinstein neighborhood $(W,\la)\sse(S^{2n+1},\xi_\st)$ of the given Legendrian $\La\sse(S^{2n+1},\xi_\st)$, thus identifying the contact $\Z_2$-cyclic branched cover $(C_2(\tau(\La)),\xi_2(\tau(\La))$ with the Weinstein sum of $(S^{2n+1},\xi_{\st})$ with a copy of itself along the Weinstein hypersurface $(W,\la)$.

Second, we now follow the proof of Proposition \ref{prop:weinsteinsumcob} in the case of $(W,\la)\cong(T^*\La,\la_\st)$. The Proposition provides a $(2n+2)$-dimensional Weinstein cobordism $(X,\la)$ with concave end the disjoint union of two copies of $(S^{2n+1},\xi_{std})$ and convex end $(C_2(\tau(\La)),\xi_2(\tau(\La))$, the Legendrian handle decomposition of which we now describe.

The Weinstein neighborhood $(T^*\La,\la_\st)$ of Legendrian sphere $\La\sse(S^{2n+1},\xi_\st)$ has a Legendrian handlebody decomposition \cite[Section 2]{CasalsMurphy} given by attaching a $2n$-dimensional Weinstein $n$-handle to the disk $(D^{2n},\la_\std)$ along the Legendrian unknot $\La_0\sse(\partial D^{2n},\ker\{\la_\st|_{\partial D^{2n}}\})\cong(S^{2n+1},\xi_\st)$. The use of the Weinstein cobordism $(W^{(1)},\la^{(1)})$ in the proof of Proposition \ref{prop:weinsteinsumcob} dictates that the Weinstein cobordism $(X,\la)$ is constructed from the symplectization of two copies of $(S^{2n+1},\xi_{\st})$ by first attaching a $2n$-dimensional Weinstein $1$-handle with attaching isotropic sphere $S^0\sse\La\sqcup\La\sse(S^{2n+1},\xi_\st)\sqcup(S^{2n+1},\xi_\st)$, where the first point in $S^0$ is attached to the first factor of the disjoint union, and the second point of $S^0$ is attached to the second factor of the disjoint union. By construction of the contact connected sum $(S^{2n+1},\xi_\st)\#(S^{2n+1},\xi_\st)$, the convex end of the Weinstein cobordism induced by attaching this Weinstein $1$-handle is contactomorphic to the contact connected sum.

In the identification of the convex boundary with $(S^{2n+1},\xi_\st)\#(S^{2n+1},\xi_\st)$, the two Legendrians $\La\sqcup\La$ in the concave end become the Legendrian connected sum $\La\#\La$, as defined in Subsection~\ref{sssec:legsum}. Indeed, the Weinstein $1$--handle is attached to an isotropic $S^0$ in $\La\sqcup\La$. The Weinstein cobordism $(X,\la)$ is obtained by attaching a $(2n+2)$--dimensional Weinstein $(n+1)$--handle along the Legendrian sphere $\La\#\La$. Let us consider the Weinstein filling $(D^{2n+2},\la_\st)\sqcup(D^{2n+2},\la_\st)$ of the concave end of $(X,\la)$, thus completing $(X,\la)$ to a Weinstein filling $(\overline{X},\overline{\la}_\st)$. The Legendrian handlebody decomposition for $(\overline{X},\overline{\la}_\st)$ induced by the above construction has two Weinstein $0$-handles, a unique Weinstein $1$-handle and a unique critical Weinstein handle. This is depicted in the lower row of Figure~\ref{fig:Z0ContactSurgery}.

The statement of Theorem \ref{thm:contactsurgery} now follows by modifying the above Legendrian handle decomposition by applying a cancellation move \cite[Proposition 2.21]{CasalsMurphy} to the canceling pair given by one of the Weinstein $0$-handle and the Weinstein $1$-handle.
\end{proof}

\section{Non-isotopic isocontact embeddings}
%\section{Proof of Theorem \ref{thm:main}}
\label{sec:mainproof}
Let us construct two isocontact embeddings using contact push-offs of Legendrian spheres in $(S^{2n+1},\xi_\std)$. For the first embedding, consider the standard Legendrian unknot $\Lambda_0 \sse (S^{2n+1},\xi_\std)$. The contact push-off $\tau(\La_0)$ defines a contact embedding
$$i_0:(\dd T^*S^n,\xi_\std)\lr(S^{2n+1},\xi_\std).$$

For the second isocontact embedding, consider a Legendrian stabilization $\La_1:=s(\La_0)$ of $\La_0\sse(S^{2n+1},\xi_\st)$. The contact push-off $\tau(\La_1)$ of this stabilized Legendrian unknot yields our second contact embedding
$$i_1:(\dd T^*S^n,\xi_\std)\sse(S^{2n+1},\xi_\std).$$

\begin{proof}[Proof of Theorem~\ref{thm:main}]
We must now show that these two isocontact embeddings
$$i_0,i_1:(\dd T^*S^n,\xi_\std)\sse(S^{2n+1},\xi_\std),$$
are formally contact isotopic but not contact isotopic, as stated in Theorem~\ref{thm:main}. To show that $i_0,i_1$ are formally contact homotopic, we apply Lemma \ref{lem:formallegisotopy} to deduce that $\La$ and $s(\La)$, which then allows us to apply our Lemma~\ref{lem:formaldetermined}, proving that the isocontact embeddings $i_0$ and $i_1$ are formally contact isotopic.

Suppose now, by contradiction, that $i_0$ and $i_1$ are contact isotopic. Then the branched double covers of $(S^{2n+1},\xi_\std)$ with branch loci $\tau(\La_0)$ and $\tau(\La_1)$ would be contactomorphic, and hence have the same symplectic fillings. Thus Theorem \ref{thm:main} shall follow from the following two propositions. 

\begin{prop}\label{prop:Z0}
The contact manifold $(C_2(\tau(\La_0)),\xi_2(\tau(\La_0))$ admits the adapted contact open book decomposition
$$(C_2(\tau(\La_0)),\xi_2(\tau(\La_0))=\ob((T^*S^{n-1},\la_\std,\tau^2_{S^{n-1}})).$$
In particular, $(C_2(\tau(\La_0)),\xi_2(\tau(\La_0))$ does not admit a flexible Weinstein filling.
\end{prop}

\begin{proof}
%[Proof of Proposition~\ref{prop:Z0}]
The standard contact $(2n+1)$-sphere admits \cite[Section 4.3]{Colin08} the contact open book decomposition
	$$(S^{2n+1},\xi_\std)=\ob((T^*S^{n},\la_\std),\tau_{S^{n-1}})).$$
	The Legendrian lift of the exact Lagrangian zero-section $S^{n}\sse(T^*S^{n},\la_\std)$ is Legendrian isotopic to the Legendrian unknot $\La_0$ in $(S^{2n+1},\xi_\std)$, see for instance \cite[Section 2]{CasalsMurphy}. In consequence, the Weinstein page of the open book can be taken to be the Weinstein neighborhood $\La_0\sse(S^{2n+1},\xi_\std)$, and hence the binding of the open book is the contact push-off $\tau(\La_0)$. This implies that $(C_2(\tau(\La_0)),\xi_2(\tau(\La_0))$ is the contact double branched cover of
	$$(S^{2n+1},\xi_\std)=\ob((T^*S^{n},\la_\std),\tau_{S^{n-1}})),$$
	branched along the contact binding. These particular contact branched covers admit a direct contact open book decomposition. Indeed, one may readily verify that the contact $k$-cyclic branched cover of the contact manifold $\ob((W,\la),\varphi))$ branched along the binding is supported by the open book $\ob((W,\la),\varphi^k))$, see \cite[Section 4]{KwonVanKoert16} and \cite[Section 4]{OzturkNiederkruger07} for details. In particular, we conclude that $(C_2(\tau(\La_0)),\xi_2(\tau(\La_0))$ is supported by the contact open book $\ob((T^*S^{n},\la_\std),\tau^2_{S^{n-1}}))$, which proves the first statement in Proposition~\ref{prop:Z0}.
	
	Regarding the second statement, $(C_2(\tau(\La_0)),\xi_2(\tau(\La_0))$ is a Brieskorn manifold, see for example \cite{OzturkNiederkruger07}, and bounds a Brieskorn variety $(V,\omega)$. More specifically our contact manifold is oftentimes called the contact Brieskorn sphere $\Sigma(2,\ldots,2)$ in the literature \cite{KwonVanKoert16,OzturkNiederkruger07,Zhou} and the filling is simply $(T^*S^{n+1}, \la_\std)$. That this indeed fills our contact manifold can be deduced by noting that $(T^*S^{n+1},\la_\std)$ admits a Weinstein Lefschetz fibration \cite[Section 3]{CasalsMurphy} with Weinstein fiber $(T^*S^{n},\la_\std)$ and two critical points, whose vanishing cycles are each a copy of the Lagrangian zero section $S^n\sse(T^*S^{n},\la_\std)$ in a regular fiber. Now the Lefschetz fibration restricts to an open book on $\partial T^*S^{n+1}$ that agrees with our given open book decomposition. 
	
	Kwon and van Koert \cite[Theorem 1.2]{KwonVanKoert16} computed that the symplectic homology of the Brieskorn variety $(V,\omega)$ and showed, in particular, that it does not vanish
	$$SH_*((V,\omega);\Z)\neq0.$$ 
	On the other hand, \cite[Theorem 20]{Zhou} shows that if a contact manifold admits a flexible filling then any Weinstein filling of that contact manifold will have vanishing symplectic homology. This result, which initially used \cite{BourgeoisEkholmEliashberg12}, also follows, independently, from the h-principle \cite[Section 6.2]{EliashbergMurphy13}. This follows from the arguments in the proof of Theorem~3.2 in \cite{MurphySiegel18}. 
\end{proof}

\begin{prop}\label{prop:flexiblefilling}
The contact structure $(C_2(\tau(\La_1)),\xi_2(\tau(\La_1))$ in the $2$-fold contact branched cover of $(S^{2n+1},\xi_\std)$ along $\tau(\La_1)$ admits a flexible Weinstein filling.
\end{prop}

\begin{proof}
%[Proof of Proposition~\ref{prop:flexiblefilling}]
Theorem~\ref{thm:contactsurgery} implies that a Weinstein filling for our contact manifold
$$(C_2(\tau(\La_1)),\xi_2(\tau(\La_1)))$$ is obtained by attaching a Weinstein critical $(n+1)$-handle to $(D^{2n+2},\la_\std)$ along the Legendrian $\La_1\#\La_1\sse(S^{2n+1},\xi_\std)$. The Legendrian connected sum of two loose Legendrians is a loose Legendrian, and thus the Weinstein filling is flexible.
\end{proof}

This concludes the proof of Theorem \ref{thm:main} and thus shows that there exists one rigid pair of contact submanifolds in all higher-dimensions.\end{proof}

\section{Infinite pairs of non-isotopic contact submanifolds}\label{sec:infinte}
The infinite family of pairs of non-contact isotopic embeddings in Theorem~\ref{cor:infinitelymany} will be constructed by studying the contact push-offs of a class of singular Legendrians and then proven distinct just as we distinguished the two examples from our Theorem~\ref{thm:main} in the previous section. 

\subsection{An infinite family of non-isotopic isocontact embeddings}
Let $(A^{2n}_k,\la_\std,\p_\std)$ be the Weinstein manifold obtained as an $A_k$--linear plumbing of $k$ copies of the Weinstein manifold $(T^*S^{n},\la_\std,\p_\std)$. See Figure~\ref{AkFiber} for a handle presentation. This manifold is known as the $2n$--dimensional Milnor fibre of the $A_k$ singularity, see Subsection~\ref{sssec:obstab} above, and its Lagrangian skeleton $\Sk(A^{2n}_k)\sse(A^{2n}_k,\la_\std,\p_\std)$ consists of a linear $A_k$-plumbing \cite[Section~3]{CasalsMurphy} of Lagrangian $n$--spheres. 
%Let us show that the Legendrian lift of this exact Lagrangian skeleton $\Sk(A^{2n}_k)$ admits a Legendrian embedding in $(S^{2n+1},\xi_\std)$ as shown in the right of Figure~\ref{LagSkel}. 
%
%\begin{figure}[ht]{\tiny
%\begin{overpic}%[grid,tics=10] 
%{LagSkel.ps}
%%\put(120, 120){$+$}
%\end{overpic}}
%\caption{The left hand side gives a handle presentation for $A_k^{4}$, and the Legendrian lift of its Lagrangian skeleton on the right.}
%\label{LagSkel}
%\end{figure}

The standard contact sphere $(S^{2n+1},\xi_\std)$ admits the contact open book decomposition
$$(S^{2n+1},\xi_\std)\cong\ob((A^{2n}_k,\la_\std,\p_\std),\tau_{A_k}),$$
where $\tau_{A_k}\in\Symp^c(A^{2n}_k,d\la_\std)$ is the compactly supported symplectomorphism obtained by composing, in any order, the symplectic Dehn twists \cite{Arnold95,Seidel08} along each of the $k$ plumbed zero-sections $S^n\sse(T^*S^{n},\la_\std,\p_\std)$ in the $A_k$-Milnor fiber. Indeed, this is the $k$-fold stabilization of the contact open book decomposition $(S^{2n+1},\xi_\std)=\ob(D^{2n},\id)$, as described in Subsection~\ref{sssec:obstab}, which by Theorem~\ref{thm:obstab} is contactomorphic to $(S^{2n+1},\xi_\std)$. Alternatively, this can be seen by considering the Milnor fibration associated to the $(n+1)$-dimensional $A_k$-singularity \cite{Arnold90,CasalsMurphy}.

The contact push-off associated to the Weinstein hypersurface $$(A^{2n}_k,\la_\std,\p_\std)\sse \ob((A^{2n}_k,\la_\std,\p_\std),\tau_{A_k})$$
given by the page of this contact open book is the contact binding $\dd(A^{2n}_k,\la_\std)$. This yields the desired embedding of a Weinstein neighborhood of $\Sk(A^{2n}_k)\sse(A^{2n}_k,\la_\std,\p_\std)$ in $(S^{2n+1},\xi_\std)$, i.e. in the language of Section \ref{sec:pushoff}, $\tau(\Sk(A^{2n}_k))=\dd(A^{2n}_k,\la_\std)$.

The two infinite families of isocontact embeddings are
$$i^k_0:\tau(\Sk(A^{2n}_k))\lr(S^{2n+1},\xi_\std),$$
$$i^k_1:\tau(s(\Sk(A^{2n}_k)))\lr(S^{2n+1},\xi_\std),$$
where $s(\Sk(A^{2n}_k))$ is the stabilized Legendrian lift of the Lagrangian skeleton $\Sk(A^{2n}_k)$, as constructed in Section \ref{sec:pushoff}. 

\subsection{Distinguishing isocontact embeddings}
%\subsection{Proof of Theorem~\ref{cor:infinitelymany}}
Let us now proceed with the proof of Theorem~\ref{cor:infinitelymany}.
\begin{proof}[Proof of Theorem~\ref{cor:infinitelymany}]
Let $n\in\N$ be fixed and $n\geq2$. Following the strategy of the proof of Theorem~\ref{thm:main} in Section~\ref{sec:mainproof}, we first note that Lemma~\ref{lem:singpushoff} shows that $i^k_0$ and $i^k_1$ are formally contact isotopic for all $k\in\N$. Let us show that $(i_0^k,i_1^k)$, with $k\in\N$, is an infinite family of pairs of non-isotopic contact embeddings in the standard contact sphere $(S^{2n+1},\xi_\std)$. There are two steps:

\begin{itemize}
	\item[-] For a fixed $k\in\N$, the isocontact embeddings $i_0^{k}$ and $i_1^{k}$ are not contact isotopic.
	
	\item[-] For any two $k,l\in\N$, $i_0^{k}$ is not contact isotopic to $i_0^{l}$ nor $i_1^{l}$.
\end{itemize}

Let us first address the second item by showing that the contact domains of $i_0^{k}$ and $i_0^{l}$, and thus also that of $i_1^{l}$, are not contactomorphic. Consider the two cases, where $n$ is either even or odd. In the former case, $n$ even, the smooth manifolds $\tau(\Sk(A^{2n}_k))$ have $(n-1)$st Betti number
$$\mbox{rk}(H_{n-1}(\tau(\Sk(A^{2n}_k)),\Z))=k-1,$$
as proven in \cite[Section 2]{Randell} and thus the domains are not homotopy equivalent unless $k=l$. In the later case, $n$ odd, the contact domains can be diffeomorphic for different $k,l\in\N$. Nevertheless, the mean Euler characteristic is a contact invariant in this case by \cite[Lemma~5.15]{KwonVanKoert16}, or \cite[Corollary~2.2]{FrauenfelderSchlenvanKoert}. For $k$ odd, they computed it to be
$$\chi_{m}(\tau(\Sk(A^{2n}_{2s+1})))=\frac{1}{2}\frac{(n-1)(2s+1)+1}{(n-2)(2s+1)+2},$$
which is an injective function on $k=2s+1\in\N$, and thus the contact domains, for $k$ odd, are not contactomorphic. For $k\in\N$ even, $k=2s$, the mean Euler characteristic is $\chi_{m}(\tau(\Sk(A^{2n}_{2s+1})))=1$ by \cite[Section 5]{KwonVanKoert16}. In fact, for $k$ even, the positive $S^1$-equivariant symplectic homologies of the Brieskorn fillings coincide. Thus, this is not enough to distinguish the contactomorphism type of the contact push-offs. Nonetheless, \cite[Lemma 3.2]{Uebele} shows that the full positive symplectic homology is a contact invariant if computed for a Weinstein filling with vanishing first Chern class. Since $c_1(A^{2n}_k,\la_\std)=0$ for all $k\in\N$, given that $(A^{2n}_k,\la_\std)$ is an affine hypersurface, the positive symplectic homology of the Brieskorn Milnor fiber $(A^{2n}_k,\la_\std)$ is a contact invariant of its convex boundary $\tau(\Sk(A^{2n}_{2s}))$. The positive symplectic homology of $(A^{2n}_{2s},\la_\std)$ is computed in \cite[Theorem 3.1]{Uebele}, where it is shown to be distinct for different values of $s\in\N$ \cite[Corollary 3.3]{Uebele}. This establishes the second item above.

As in Section~\ref{sec:mainproof}, we distinguish the two isocontact embeddings $i_0^{k}$ and $i_1^{k}$ in the first item above by the contactomorphism type of their cyclic 2--fold branched covers, which in turn are distinguished by the existence --- and lack thereof --- of flexible Weinstein fillings. The first item, and thus Theorem~\ref{cor:infinitelymany} follow from the following result. 
\end{proof}

\begin{prop}\label{prop:Ak}
	Let $n\geq2$ and $k\in\N$, and consider the two isocontact embeddings
	$$i^k_0:\tau(\Sk(A^{2n}_k))\lr(S^{2n+1},\xi_\std),$$
	$$i^k_1:\tau(s(\Sk(A^{2n}_k)))\lr(S^{2n+1},\xi_\std).$$
	
	Then the cyclic 2--fold branched cover of $(S^{2n+1},\xi_\std)$ along the contact submanifold $\im(i^k_0)$ does not admit a flexible Weinstein filling. In contrast,  the cyclic 2--fold branched cover of $(S^{2n+1},\xi_\std)$ along  the contact submanifold $\im(i^k_1)$ does.\hfill$\Box$
\end{prop}

The first statement in Proposition \ref{prop:Ak} is proven exactly as Proposition \ref{prop:Z0}, since \cite[Theorem~1.2]{KwonVanKoert16} and \cite[Theorem~20]{Zhou} apply to all Weinstein fillings of Brieskorn manifolds. The second statement follows from Theorem~\ref{thm:contactsurgery} by the same argument as in Proposition~\ref{prop:flexiblefilling}. 

%\nocite{*}
%\bibliographystyle{gtart}
\bibliographystyle{plain}
\bibliography{CasalsEtnyre}
\end{document}